\documentclass[draft,10pt,a4paper,reqno]{amsart}

\usepackage{pdfsync}

\usepackage[latin1]{inputenc}
\usepackage{amsmath,amsfonts,amssymb}
\usepackage{amscd}
\usepackage{graphicx}

\usepackage{mathrsfs}

\DeclareMathAlphabet{\mathpzc}{OT1}{pzc}{m}{it}

\usepackage[all]{xy}

\entrymodifiers={+!!<0pt,\fontdimen22\textfont2>}

\font\sss=cmss8

\def\BZ{{\mathbb Z}}

\def\sC{\mbox{\sf C}}
\def\sD{\mbox{\sf D}}

\def\ssC{\mbox{\sss C}}

\def\add{\operatorname{add}}

\def\adots{\mathinner{\mkern1mu\raise1.0pt\vbox{\kern7.0pt\hbox{.}}\mkern2mu\raise4.0pt\hbox{.}\mkern2mu\raise7.0pt\hbox{.}\mkern1mu}}

\def\D{\mbox{\sD}}

\def\Df{\D^{\operatorname{f}}}

\def\End{\operatorname{End}}

\def\Ext{\operatorname{Ext}}

\def\Hom{\operatorname{Hom}}

\def\inf{\operatorname{inf}}

\def\mod{\mbox{\sf mod}}

\def\modulo{\operatorname{mod}}

\def\prod{\operatorname{prod}}

\def\stab{\mbox{\sf stab}}

\numberwithin{equation}{part}

\newtheorem{Lemma}{Lemma}[section]
\newtheorem{Theorem}[Lemma]{Theorem}
\newtheorem{Proposition}[Lemma]{Proposition}

\theoremstyle{definition}

\newtheorem{Remark}[Lemma]{Remark}

\newtheorem{Example}[Lemma]{Example}

\begin{document}

\setlength{\parindent}{0pt}
\setlength{\parskip}{7pt}

\title[Cluster categories and stable module categories]
{Realising higher cluster categories of Dynkin type 
as stable module categories}

\author{Thorsten Holm}
\address{Institut f\"{u}r Algebra, Zahlentheorie und Diskrete
  Mathematik, Fakult\"at f\"ur Mathematik und Physik, Leibniz
  Universit\"{a}t Hannover, Welfengarten 1, 30167 Hannover, Germany}

\email{holm@math.uni-hannover.de}
\urladdr{http://www.iazd.uni-hannover.de/\~{ }tholm}

\author{Peter J\o rgensen}
\address{School of Mathematics and Statistics,
Newcastle University, Newcastle upon Tyne NE1 7RU,
United Kingdom}
\email{peter.jorgensen@ncl.ac.uk}
\urladdr{http://www.staff.ncl.ac.uk/peter.jorgensen}





\thanks{{\em Acknowledgement. }This work was carried out in the
  framework of the research priority programme SPP 1388 {\em
    Darstellungstheorie} of the Deutsche Forschungsgemeinschaft (DFG).
  We gratefully acknowledge financial support through the grant HO
  1880/4-1. }


\begin{abstract}
  
  We show that the stable module categories of certain selfinjective
  algebras of finite representation type having tree class $A_n$, $D_n$,
  $E_6$, $E_7$ or $E_8$ are triangulated equivalent to $u$-cluster 
  categories of the corresponding Dynkin type.
  The proof relies on the ``Morita'' theorem for $u$-cluster categories by
  Keller and Reiten, along with the recent computation of Calabi-Yau
  dimensions of stable module categories by Dugas. 

\end{abstract}

\maketitle

\section{Introduction}
\label{sec:introduction}

This paper deals with two types of categories: {\em Stable module
categories of selfinjective algebras} and {\em $u$-cluster
categories}.  They both originate in representation theory, and we
will establish a connection between the two by showing that some
stable module categories are, in fact, $u$-cluster categories.

{\em Stable module categories} are classical objects of representation
theory.  They arise from categories of finitely generated modules
through the operation of dividing by the ideal of homomorphisms which
factor through a projective module.  The stable module category of a
finite dimensional selfinjective algebra has the appealing property
that it is triangulated; this has been very useful not least in group
representation theory.

{\em Cluster categories} and the more general {\em $u$-cluster
  categories} which are pa\-ra\-me\-tri\-sed by the natural number $u$
were introduced over the last few years in a number of beautiful
papers: \cite{BMRRT}, \cite{CCS}, \cite{Keller}, \cite{Thomas}, and
\cite{BinZhu}.  
The idea is to provide ca\-te\-go\-ri\-fi\-ca\-ti\-ons of the theory
of cluster algebras and higher cluster complexes as introduced in
\cite{FR} and \cite{FZ}.  If $Q$ is a finite quiver without loops and
oriented cycles, then the $u$-cluster category of type $Q$ over a
field $k$ is defined by considering the bounded derived category of
the path algebra $kQ$ and taking the orbit category of a certain
autoequivalence; see Section \ref{sec:cluster} for details.  A
$u$-cluster category is triangulated; this non-trivial fact was
established in \cite{Keller}.

The introduction of cluster categories and $u$-cluster categories has
created a rush of activity which has turned these categories into a
major item of contemporary representation theory.  This is due not
least to the advent of cluster tilting theory in \cite{BMR}, which
provides a long awaited generalization of classical tilting theory
making it possible to tilt at any vertex of the quiver of a
hereditary algebra, not just at sinks and sources.

In this paper, we will show that a number of stable module categories
of selfinjective algebras are, in fact, $u$-cluster categories.

To be precise, we will look at stable module categories of
selfinjective algebras of finite representation type. By the Riedtmann
structure theorem \cite{Riedtmann2} the Auslander-Reiten (AR) quiver
of such a category has tree class of Dynkin type $A_n$, $D_n$,
$E_6$, $E_7$, or $E_8$.  We illustrate in type $A$ what this means.
Consider the Dynkin quiver in Figure \ref{fig:Dynkin_A}
which, by abuse of notation, we will often denote by $A_n$, and its
repetitive quiver $\BZ A_n$ shown in Figure \ref{fig:ZA}.
\begin{figure}
\[
  \xymatrix{ 1 \ar[r] & 2 \ar[r] & \cdots \ar[r] & n }
\]
\label{fig:Dynkin_A}
\caption{The Dynkin quiver $A_n$}
\end{figure}
\begin{figure}
\[
  \xymatrix @+0.9pc @!0 {
& & & & {} \ar[dr] & & n \ar[dr] & & *{\circ} \ar[dr] & \\
& & & & & n-1 \ar[dr] \ar[ur] & & *{\circ} \ar[dr] \ar[ur] & & \ddots \\
& & \ddots \ar[dr] & & \adots \ar[ur] & & *{\circ} \ar[dr] \ar[ur] & & \ddots & \\
& \ddots \ar[dr] & & 3 \ar[dr] \ar[ur] & & \adots \ar[ur] & & \ddots & & \\
\ddots \ar[dr] & & 2 \ar[dr] \ar[ur] & & *{\circ} \ar[dr] \ar[ur] & & & & & \\
& 1 \ar[ur] & & *{\circ} \ar[ur] & & {} & & & & \\
                      }
\]
\caption{The repetitive quiver $\BZ A_n$}
\label{fig:ZA}
\end{figure}
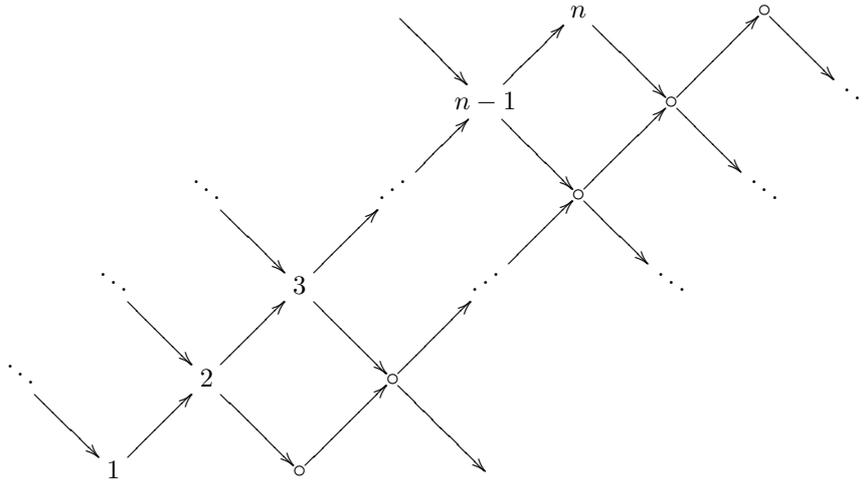
For a selfinjective algebra to have finite representation type and
tree class $A_n$ means that the AR quiver of its stable
module category is a non-trivial quotient of $\BZ A_n$ by an admissible
group of automorphisms. In type $A$, in such a
quotient, two vertical lines on the quiver are identified, and this
gives either a cylinder or a M\"{o}bius band.  According to this
dichotomy, the algebra belongs to one of two well understood classes:
the Nakayama algebras and the M\"{o}bius algebras.

For tree classes $D_n$ and $E_6$, $E_7$, $E_8$, the shapes of the
stable AR quivers are obtained in a very similar fashion; more details
on the precise shapes are given in Section \ref{sec:typeD} for type
$D$ and Section \ref{sec:typeE} for type $E$.

Now, $u$-cluster categories of Dynkin types $A_n$, $D_n$, $E_6$,
$E_7$, and $E_8$ also have AR quivers which are either cylinders or
M\"{o}bius bands; see Section \ref{sec:cluster} for details. One of
the aims of this paper is to show that this resemblance is no
coincidence.

For stating the main results of the paper we have to deal
with the various Dynkin types separately. 

Let us start with Dynkin type $A$. For integers $N, n \geq 1$, let
$B_{N,n+1}$ denote the Nakayama algebra defined as the path algebra of
the circular quiver with $N$ vertices and all arrows pointing in the
same direction modulo the ideal generated by all paths of length
$n+1$.  Moreover, for integers $p, s \geq 1$, let $M_{p,s}$ denote the
corresponding M\"obius algebra (for the definition of these algebras
by quivers and relations, see Section \ref{sec:Mobius}).

The following is our first main result which gives a complete list of
those $u$-cluster categories of type $A$ which are triangulated
equivalent to stable module categories of selfinjective algebras.

\smallskip

\noindent
{\bf Theorem A} (Realising $u$-cluster categories of type $A$).
{\it 
\begin{enumerate}

\item  Let $u \geq 2$ be an even integer and let $n \geq 1$ be an
integer. Set
$N = \frac{u}{2}(n+1) + 1.$
Then the $u$-cluster category of type $A_n$ is equivalent as a
triangulated category to the stable module category $\stab\,
B_{N,n+1}$.  

\medskip

\item
Let $u \geq 1$ be an odd integer and let $p,s \geq 1$ be
integers for which
$s(2p+1) = u(p+1) + 1.$
Then the $u$-cluster category of type $A_{2p+1}$ is equivalent as a
triangulated category to the stable module category
$\stab\, M_{p,s}$.

\end{enumerate}
}

\smallskip

We next consider Dynkin types $D$ and $E$.  The theory becomes more
intricate than in type $A$.  While two types of selfinjective algebras
occurred in type $A$, we will show that three types of algebras occur
in type $D$, and two in type $E$.  More precisely, in Asashiba's
notation from \cite[appendix]{Asashiba}, they are the algebras
$(D_n,s,1)$, $(D_n,s,2)$, and $(D_{3m},\frac{s}{3},1)$ in type $D$,
and $(E_n,s,1)$, $n=6,7,8$, and $(E_6,s,2)$ in type $E$.

Specifically, we show the following main results.

\smallskip

\noindent
{\bf Theorem D} (Realising $u$-cluster categories of type $D$).
{\it 
Let $m, n, u$ be integers with $u \geq 1$.
\begin{enumerate}

\item Suppose that $n\ge 4$ and $u\equiv -2 \,\modulo\, (2n-3)$. 

\medskip
\noindent
Then the $u$-cluster category of type $D_{n}$
is equivalent as a triangulated category to the stable module category
\[
  \left\{
    \begin{array}{ll}
      \stab\, (D_{n},\frac{u(n-1)+1}{2n-3},1) & \mbox{if $n$ or $u$ is
        even, } \\[2mm]
      \stab\, (D_{n},\frac{u(n-1)+1}{2n-3},2) & \mbox{if $n$ and $u$
        are odd. }
    \end{array}
  \right.
\]

\medskip

\item
Suppose that $m\ge 2$ and $u \equiv -2 \,\modulo\, (2m-1)$ but $u
\not\equiv -2 \,\modulo\, (6m-3)$.  Moreover suppose that not both $m$
and $u$ are odd. Then the $u$-cluster category of type $D_{3m}$ is
equivalent as a triangulated category to the stable module category $
\stab\, (D_{3m},\frac{s}{3},1) $ where $s=\frac{u(3m-1)+1}{2m-1}$.

\end{enumerate}
}

\noindent
{\bf Theorem E} (Realising $u$-cluster categories of type $E$).
{\it 
Let $u \geq 1$ be an integer.
\begin{enumerate}

\item If $u\equiv -2 \,\modulo\, 11$ then the $u$-cluster category of
  type $E_6$ is equivalent as a triangulated category to the stable
  module category
\[
  \left\{ 
    \begin{array}{ll}
      \stab(E_6,\frac{6u+1}{11},1) & \mbox{if $u$ is even,} \\[2mm]
      \stab(E_6,\frac{6u+1}{11},2) & \mbox{if $u$ is odd.} 
    \end{array}
  \right.
\]

\medskip

\item If $u\equiv -2 \,\modulo\, 17$ then the $u$-cluster category of
  type $E_7$ is equivalent as a triangulated category to the stable
  module category $\stab(E_7,\frac{9u+1}{17},1)$.

\medskip

\item If $u\equiv -2 \,\modulo\, 29$ then the $u$-cluster category of
  type $E_8$ is equivalent as a triangulated category to the stable
  module category $\stab(E_8,\frac{15u+1}{29},1)$.

\end{enumerate}
}

\smallskip

The proofs of the above theorems rely on the seminal ``Morita
theorem'' for $u$-cluster categories established by Keller and Reiten
in \cite{KellerReiten2}. The idea is to show that the stable module
categories of the relevant selfinjective algebras have very strong
formal properties in terms of their Calabi-Yau dimensions and
$u$-cluster til\-ting objects. More precisely, the Keller-Reiten
structure theorem states the following. Consider a Hom finite
triangulated category of algebraic origin (e.g.\ the stable module
category of a selfinjective algebra).  Assume that it has Calabi-Yau
dimension $u+1$ and possesses a $u$-cluster tilting object $T$ which
has hereditary endomorphism algebra $H$ and also satisfies
$\Hom(T,\Sigma^{-i}T) = 0$ for $i = 1, \ldots, u-1$ where $\Sigma$ is
the suspension functor.  Then this category is triangulated equivalent
to the $u$-cluster category of $H$.

Theorems A,\,D and E were already stated in our earlier preprints
\cite{HolmJorgensenA}, \cite{HolmJorgensenDE} which were later
withdrawn. Unfortunately there was a mistake in \cite{HolmJorgensenA},
pointed out to us by Alex Dugas, in connection with the Calabi-Yau
dimensions, and this meant there was a gap in the proofs of the main
results of \cite{HolmJorgensenA} and \cite{HolmJorgensenDE}.

In this paper we circumvent the problem and thereby provide correct
proofs of the above theorems.  This is achieved by using a recent
paper of Dugas \cite{Dugas} in which he computes the Calabi-Yau
dimensions for stable module categories of selfinjective algebras of
finite representation type.

The paper is organized as follows: Section \ref{sec:cluster} collects
the properties of $u$-cluster categories of Dynkin types $ADE$ which
we will need. Section \ref{sec:orthogonal} is a remark on
$u$-cluster tilting objects in stable module categories.  Section
\ref{sec:typeA} considers Dynkin type $A$ and proves Theorem A. This
is split into subsections \ref{sec:Nakayama} and \ref{sec:Mobius}
dealing with Nakayama algebras and M\"obius algebras; these two
situations correspond to parts (i) and (ii) of Theorem A.  Sections
\ref{sec:typeD} and \ref{sec:typeE} similarly consider Dynkin types
$D$ and $E$ and prove Theorems D and E.

Throughout, $k$ is an algebraically closed field, $A$ is a selfinjective
$k$-algebra, $\mod\,A$ denotes
the category of finitely generated right-$A$-modules, and $\stab\,A$
denotes the stable category of finitely ge\-ne\-ra\-ted
right-$A$-modules.

\section*{Acknowledgement}

We thank Claire Amiot and Bernhard Keller warmly for a number of
useful comments and suggestions to preliminary versions.

We are deeply grateful to Alex Dugas for pointing out a subtle but
serious mistake in our earlier manuscript \cite{HolmJorgensenA}, and
for many subsequent email discussions on the subject of determining
Calabi-Yau dimensions.

Work on this project started in September 2006 while the first author 
was visiting the Universit\'{e} Montpellier 2.  He thanks Claude Cibils 
for the invitation and the warm hospitality, and the R\'{e}gion
Languedoc-Roussillon and the Universit\'{e} Montpellier 2
for financial support.

\section{$u$-cluster categories}
\label{sec:cluster}

This section collects the properties of $u$-cluster categories which
we will need.

Let $Q$ be a finite quiver without loops and oriented cycles.  Consider the
path algebra $kQ$ and let $\Df(kQ)$ be the derived category of bounded
complexes of finitely generated right-$kQ$-modules.  See \cite{Happel}
for background on $\Df(kQ)$ and \cite{RVdB} for additional information
on AR theory and Serre functors.

If $u \geq 1$ is an integer, then the $u$-cluster category of type $Q$
is defined as $\Df(kQ)$ modulo the functor $\tau^{-1}\Sigma^u$, where
$\tau$ is the AR translation of $\Df(kQ)$ and $\Sigma$ the suspension.
In other words, the $u$-cluster category is the orbit category for the
action of $\tau^{-1}\Sigma^u$ on the category $\Df(kQ)$.  Denote the
$u$-cluster category of type $Q$ by $\sC$.

It follows from \cite[sec.\ 4, thm.]{Keller} that $\sC$ admits a
structure of triangulated category in a way such that the canonical
functor $\Df(kQ) \rightarrow \sC$ is triangulated.

The category $\sC$ has Calabi-Yau dimension $u + 1$ by \cite[sec.\ 
4.1]{KellerReiten2}.  That is, $n = u + 1$ is the smallest
non-negative integer such that $\Sigma^n$, the $n$th power of
the suspension functor, is the Serre functor of $\sC$.

The category $\sC$ has the same objects as the derived category
$\Df(kQ)$, so in particular, $kQ$ is an object of $\sC$.  In fact, by
\cite[sec.\ 4.1]{KellerReiten2} again, $kQ$ is a $u$-cluster tilting
object of $\sC$, cf.\ \cite[sec.\ 3]{IyamaYoshino}.  That is,
\begin{enumerate}

  \item  $\Hom_{\ssC}(kQ,\Sigma t) = \cdots = \Hom_{\ssC}(kQ,\Sigma^u t) = 0
         \; \Leftrightarrow \; t \in \add\, kQ$,

\medskip

  \item  $\Hom_{\ssC}(t,\Sigma kQ) = \cdots = \Hom_{\ssC}(t,\Sigma^u kQ) = 0
         \; \Leftrightarrow \; t \in \add\, kQ$.

\end{enumerate}
Recall that $\add\, kQ$ denotes the full subcategory of $\sC$ consisting
of direct summands of (finite) direct sums of copies of $kQ$. 

The endomorphism ring $\End_{\ssC}(kQ)$ is $kQ$ itself.

\subsection{$u$-cluster categories of Dynkin type $A$}
\label{subsec:clustera}
Let $Q$ be a Dynkin quiver of type $A_n$ for an integer $n \geq 1$.
This means that the graph obtained from $Q$ by forgetting the
orientations of the arrows is a Dynkin diagram of type $A_n$.  Recall
that the orientation of $Q$ is not important since for any two
orientations the derived categories $\Df(kQ)$ are triangulated
equivalent.  In the sequel we shall always use the linear orientation
as in Figure \ref{fig:Dynkin_A} in the introduction.
 
By \cite[cor.\ 4.5(i)]{Happel}, the AR quiver of $\Df(kQ)$ is the
repetitive quiver $\BZ A_n$; see Figure \ref{fig:ZA} in the
introduction.  Accordingly, the AR quiver of the $u$-cluster category
$\sC$ is $\BZ A_n$ modulo the action of $\tau^{-1}\Sigma^u$ by
\cite[prop.\ 1.3]{BMRRT}.

The AR translation $\tau$ of $\Df(kQ)$ acts on the quiver by shifting
one unit to the left.  Both here and below, a unit equals the distance
between two vertices which are horizontal neighbours.  Hence
$\tau^{-1}$ acts by shifting one unit to the right.

The suspension $\Sigma$ of $\Df(kQ)$ acts by reflecting in the
horizontal centre line and shifting $\frac{n+1}{2}$ units to the
right; see \cite[table p.\ 359]{MiyachiYekutieli}.  Note that this
shift makes sense for all values of $n$: If $n$ is even, then the
reflection in the horizontal centre line sends a vertex of the quiver
to a point midwise between two vertices, and the half integer shift by
$\frac{n+1}{2}$ sends this point to a vertex.

It follows that if $u$ is even, then $\tau^{-1}\Sigma^u$ acts by
shifting $\frac{u}{2}(n+1) + 1$ units to the right, and if $u$ is odd,
then $\tau^{-1}\Sigma^u$ acts by shifting $\frac{u}{2}(n+1) + 1$ units
to the right and reflecting in the horizontal centre line.

So if $u$ is even, then the AR quiver of $\sC$ has the shape of a
cylinder, and if $u$ is odd, then the AR quiver of $\sC$ has the shape
of a M\"{o}bius band.

\subsection{$u$-cluster categories of Dynkin type $D$}
\label{subsec:clusterD}
Let $Q$ be a Dynkin quiver of type $D_n$ for an integer $n \geq 4$.
Since the orientation of the quiver does not affect the derived
category, we can assume that $Q$ has the form in Figure \ref{fig:Dn}.
\begin{figure}
\[
  \xymatrix @+1.6pc @!0 {
             &          &               &                     & (n-1)^+  \\
    1 \ar[r] & 2 \ar[r] & \cdots \ar[r] & n-2 \ar[ur] \ar[dr] &     \\
             &          &               &                     & (n-1)^- \\
           }
\]
\caption{The Dynkin quiver $D_n$}
\label{fig:Dn}
\end{figure}
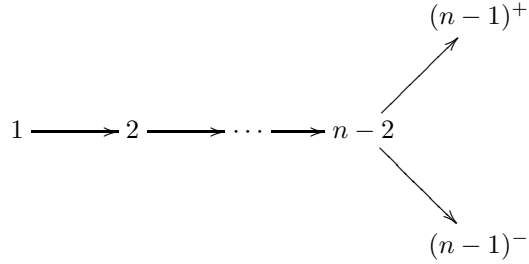
By \cite[cor.\ 4.5(i)]{Happel}, the AR quiver of $\Df(kQ)$ is the
repetitive quiver $\BZ D_n$ shown in Figure \ref{fig:ZDn}.
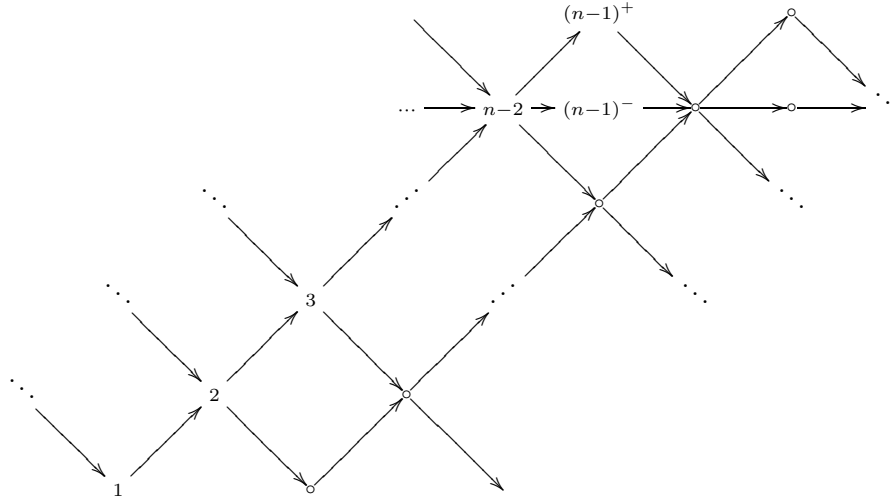
\begin{figure}
\[
  \def\objectstyle{\scriptstyle}
  \vcenter{
  \xymatrix @+1pc @!0 {
& & & & {} \ar[dr] & & (n-1)^+ \ar[dr] & & *{\circ} \ar[dr] & \\
& & & & \cdots \ar[r] & n-2 \ar[dr] \ar[ur] \ar[r] & (n-1)^- \ar[r] & *{\circ} \ar[dr] \ar[ur] \ar[r] & *{\circ} \ar[r] & \ddots \\
& & \ddots \ar[dr] & & \adots \ar[ur] & & *{\circ} \ar[dr] \ar[ur] & & \ddots & \\
& \ddots \ar[dr] & & 3 \ar[dr] \ar[ur] & & \adots \ar[ur] & & \ddots & & \\
\ddots \ar[dr] & & 2 \ar[dr] \ar[ur] & & *{\circ} \ar[dr] \ar[ur] & & & & & \\
& 1 \ar[ur] & & *{\circ} \ar[ur] & & {} & & & & \\
            }
          }
\]
\caption{The repetitive quiver $\BZ D_n$}
\label{fig:ZDn}
\end{figure}
The AR quiver of the $u$-cluster category $\sC$ is $\BZ D_n$ modulo
the action of $\tau^{-1}\Sigma^u$ by \cite[prop.\ 1.3]{BMRRT}.

Again $\tau^{-1}$ acts by shifting one unit to the right.

If $n$ is even, then the suspension $\Sigma$ acts by shifting $n-1$
units to the right, and if $n$ is odd, then $\Sigma$ acts by shifting
$n-1$ units to the right and switching each pair of `exceptional'
vertices such as $(n-1)^+$ and $(n-1)^-$; cf.\ \cite[table p.\ 
359]{MiyachiYekutieli}.

It follows that if $n$ or $u$ is even, then $\tau^{-1}\Sigma^u$ acts
by shifting $u(n-1) + 1$ units to the right, and if $n$ and $u$ are
both odd, then $\tau^{-1}\Sigma^u$ acts by shifting $u(n-1) + 1$ units
to the right and switching each pair of exceptional vertices.

Accordingly, the AR quiver of the $u$-cluster category
$\sC$ has the shape of a cylinder of circumference
$u(n-1) + 1.$

\subsection{$u$-cluster categories of Dynkin type $E$}
\label{subsec:clusterE}
Let $Q$ be a Dynkin quiver of type $E_n$ for $n = 6, 7, 8$.  We can
suppose that $Q$ has the orientation in Figure \ref{fig:En}, with the
convention that for $n = 6$ the two non-filled vertices and for $n
= 7$ the leftmost non-filled vertex (and all arrows incident to them)
do not exist.
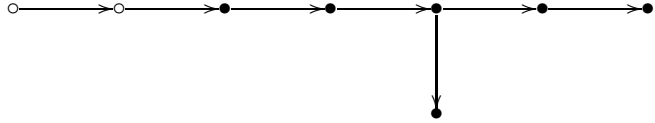
\begin{figure}
\[
  \vcenter{
  \xymatrix @+1.3pc @!0 {
    *{\circ} \ar[r] & *{\circ} \ar[r] & *{\bullet} \ar[r] & *{\bullet} \ar[r] & *{\bullet} \ar[r] \ar[d]& *{\bullet} \ar[r] & *{\bullet} \\
    & & & & *{\bullet} \\
                      }
          }
\]
\caption{The Dynkin quivers $E_6$, $E_7$, $E_8$}
\label{fig:En}
\end{figure}
By \cite[cor.\ 4.5(i)]{Happel}, the AR quiver of $\Df(kQ)$ is the
repetitive quiver $\BZ E_n$ shown in Figure \ref{fig:ZEn}.
\begin{figure}
\[
  \vcenter{
  \xymatrix @+0.65pc @!0 {
& & & & & {} \ar[dr] & & *{\bullet} \ar[dr] & & *{\bullet} \\
& & & & \ddots \ar[dr] & & *{\bullet} \ar[dr] \ar[ur] & & *{\bullet} \ar[dr] \ar[ur] & \\
& & & \ddots \ar[dr] & \cdots \ar[r] & *{\bullet} \ar[dr] \ar[ur] \ar[r] & *{\bullet} 
  \ar[r] & *{\bullet} \ar[dr] \ar[ur] \ar[r] & *{\bullet} \ar[r] & \ddots \\
& & \ddots \ar[dr] & & *{\bullet} \ar[ur]\ar[dr] & & *{\bullet} \ar[dr] \ar[ur] & & \ddots & \\
& \ddots \ar[dr] & & *{\bullet} \ar[dr] \ar[ur] & & *{\bullet} \ar[dr] \ar[ur] & & 
 \ddots & & \\
\ddots \ar[dr] & & *{\circ} \ar[dr] \ar[ur] & & *{\circ} \ar[dr] \ar[ur] & & 
  \ddots & & & \\
& *{\circ} \ar[ur] & & *{\circ} \ar[ur] & & *{\circ} & & & & \\
            }
          }
\]
\caption{The repetitive quivers $\BZ E_6$, $\BZ E_7$, $\BZ E_8$}
\label{fig:ZEn}
\end{figure}
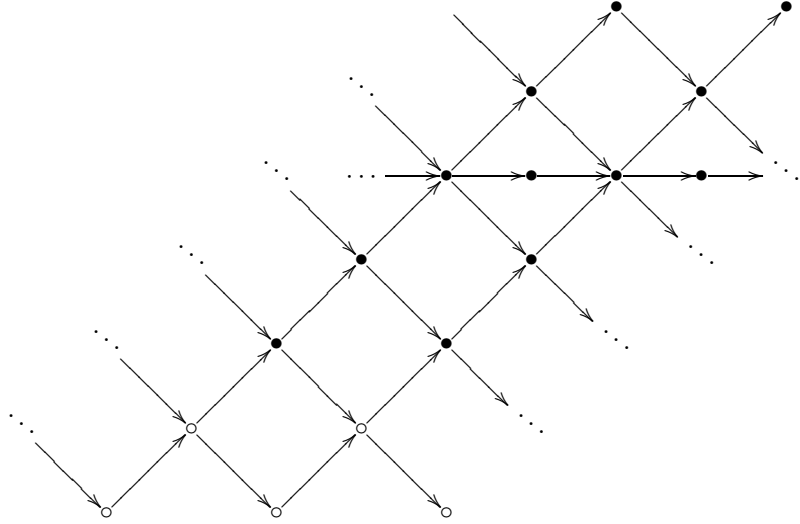
Again, for $n = 6$ and $n = 7$ the bottom two rows and bottom row,
respectively, of non-filled vertices do not occur. Note that for $n =
6$ the AR quiver has a symmetry at the central line which does not
exist for $n = 7, 8$.

The AR quiver of the $u$-cluster category $\sC$ is $\BZ E_n$ modulo
the action of $\tau^{-1}\Sigma^u$ by \cite[prop.\ 1.3]{BMRRT}.

Again $\tau^{-1}$ acts by shifting one unit to the right.

If $n = 6$ then the suspension $\Sigma$ acts by shifting 6 units to
the right and reflecting in the central line of the AR quiver.  If $n
= 7, 8$ then $\Sigma$ acts by shifting $9$, respectively $15$ units to
the right.  See \cite[table 1, p.\ 359]{MiyachiYekutieli}.

It follows that the action of $\tau^{-1}\Sigma^u$ is given as follows:
for $n = 6$ and $u$ even, by shifting $6u+1$ units to the right; for
$n = 6$ and $u$ odd, by shifting $6u+1$ units to the right and
reflecting in the central line; for $n = 7$, by shifting $9u+1$ units
to the right; for $n = 8$, by shifting $15u+1$ units to the right.

In particular, the AR quiver of the $u$-cluster category of type
$E_n$, $n = 6,7,8$, has the shape of a cylinder, except when $n = 6$ and
$u$ is odd where it has the shape of a M\"obius band.

\section{Cluster tilting objects}
\label{sec:orthogonal}

The notion of a $u$-cluster tilting object in a triangulated category
was recalled in Section \ref{sec:cluster}.  There is also a
definition in abelian categories, cf.\ \cite[sec.\ 2]{Iyama}.  An
object $X$ of an abelian category is called $u$-cluster tilting if
\begin{enumerate}

  \item  $\Ext^1(X,t) = \cdots = \Ext^u(X,t) = 0 \; \Leftrightarrow \;
         t \in \add\, X$,

\medskip

  \item  $\Ext^1(t,X) = \cdots = \Ext^u(t,X) = 0 \; \Leftrightarrow \;
         t \in \add\, X$.

\end{enumerate}

Over selfinjective algebras, there is the following simple connection
between $u$-cluster tilting objects in the module category (which
is abelian) and the stable module category (which is triangulated).

\begin{Proposition}
\label{prop:max-orth}
Let $A$ be a selfinjective $k$-algebra and let $X$ be a 
$u$-cluster tilting object of the module category $\mod\,A$.  Then $X$ is
also a $u$-cluster tilting object of the stable module category
$\stab\,A$.
\end{Proposition}

\begin{proof}
Since $A$ is selfinjective, the suspension functor $\Sigma$ provides
us with isomorphisms
\[
  \underline{\Hom}(M,\Sigma^i N) \cong \Ext^i(M,N)
\]
for $M$ and $N$ in $\mod\,A$ and $i \geq 1$.  Here
$\underline{\Hom}$ denotes morphisms in $\stab\,A$.

On one hand, this implies
\[
  \underline{\Hom}(X,\Sigma^1 X) = \cdots
  = \underline{\Hom}(X,\Sigma^u X) = 0.
\]

On the other hand, suppose that $t$ in $\stab\,A$ satisfies
\[
  \underline{\Hom}(X,\Sigma^1 t) = \cdots
  = \underline{\Hom}(X,\Sigma^u t) = 0.
\]
Then
\[
  \Ext^1(X,t) = \cdots = \Ext^u(X,t) = 0,
\]
so $t$ is in $\add\, X$ viewed in $\mod\,A$.  But then $t$ is
clearly also in $\add\, X$ viewed in $\stab\,A$.

A similar argument shows that
\[
  \underline{\Hom}(t,\Sigma^1 X) = \cdots
  = \underline{\Hom}(t,\Sigma^u X) = 0
\]
implies that $t$ is in $\add\, X$ viewed in $\stab\,A$.
\end{proof}

\section{Dynkin type $A$}
\label{sec:typeA}

\subsection{Nakayama algebras}
\label{sec:Nakayama}

This subsection proves part (i) of Theorem A from the introduction. 

For integers $N, n \geq 1$, consider the Nakayama algebra $B_{N,n+1}$
defined as the path algebra of the circular quiver with $N$ vertices
and all arrows pointing in the same direction, modulo the ideal
generated by paths of length $n+1$.

This is a selfinjective algebra of tree class $A_n$.  The stable AR
quiver of $B_{N,n+1}$ has the shape of a cylinder and can be obtained
as $\BZ A_n$ modulo a shift by $N$ units to the right.

On the other hand, as we saw in Section \ref{sec:cluster}, if $u$ is
even then the $u$-cluster category of type $A_n$ has an AR quiver
which can be obtained as $\BZ A_n$ modulo a shift by $\frac{u}{2}(n+1)
+ 1$ units to the right.  Indeed, this is no coincidence.

\begin{Theorem}
\label{thm:Nakayama}
Let $u \geq 2$ be an even integer and let $n \geq 1$ be an integer.
Set
\[
  N = \frac{u}{2}(n+1) + 1.
\]
Then the $u$-cluster category of type $A_n$ is equivalent as a
triangulated category to the stable module category $\stab\,
B_{N,n+1}$.
\end{Theorem}

\begin{proof}
For $n=1$ the theorem states that the $u$-cluster category of type
$A_1$ is triangulated equivalent to $\stab\,B_{u+1,2}$. This is true
by the observation that both categories have AR quiver a disconnected
union of $u+1$ vertices, with suspension equal to a cyclic shift by
one vertex.

We now assume $n\ge 2$, in which case the relevant categories are
connected.  By Keller and Reiten's Morita theorem for $u$-cluster
categories \cite[thm.\ 4.2]{KellerReiten2}, we need to show three
things for the stable module category $\stab\, B_{N,n+1}$.
\begin{itemize}

  \item  It has {\em Calabi-Yau dimension } $u + 1$. 

\medskip

  \item  It has {\em a $u$-cluster tilting object } $X$ with
    endomorphism ring $kA_n$.

\medskip

  \item  The object $X$ has {\em vanishing of negative self-extensions
    } in the sense that
$\underline{\Hom}(X,\Sigma^{-i}X) = 0$ for $i=1,
\ldots, u-1$.
\end{itemize}
According to this, the proof is divided into three sections.  Note
the shift in the indices compared to \cite{KellerReiten2}: their
$d$-cluster categories are $u$-cluster categories for $u=d-1$ in our
notation.

\medskip
{\em Calabi-Yau dimension. }
We must show that $\stab\, B_{N,n+1}$ has Calabi-Yau dimension $u +
1$, and we can do so using the results by Dugas in \cite{Dugas}.  To
apply his result from \cite[thm.\ 6.1(2)]{Dugas} in our case of type
$A_n$ where $n\ge 2$, we need the Coxeter number $h_{A_n}=n+1$, and we
have to observe that in Asashiba's notation from
\cite[appendix]{Asashiba} the Nakayama algebra $B_{N,n+1}$ has the
form $(A_n,\frac{N}{n},1)$ where $f=\frac{N}{n}$ is the {\em
  frequency}. Then \cite[thm.\ 6.1(2)]{Dugas} states that the stable
module category $\stab\, B_{N,n+1}$ has Calabi-Yau dimension $2r+1$
where $r \equiv -(h_{A_n})^{-1} \,\modulo\, fn$ and $0\le r<fn$.  Since
$f=\frac{N}{n}$ the value of $r$ is determined by $0\le r<N$ and $r
\equiv -(h_{A_n})^{-1} \,\modulo\, N = -(n+1)^{-1} \,\modulo\, N$.
By our assumptions in Theorem \ref{thm:Nakayama} we have that $u = 2\ell$
is even and that $N = \frac{u}{2}(n+1)+1 = \ell(n+1)+1$. Then the condition
for the value of $r$ reads $r \equiv
-(n+1)^{-1} \,\modulo\, (\ell(n+1)+1)$ which together with $0\le r<
N = \ell(n+1)+1$ clearly forces $r = \ell$.  Therefore we can deduce that
$\stab\, B_{N,n+1}$ has Calabi-Yau dimension $2r+1 = 2\ell+1 = u+1,$ as
desired.

\medskip
{\em $u$-cluster tilting object. }
To find a $u$-cluster tilting object $X$ in $\stab\,B_{N,n+1}$, by
Proposition \ref{prop:max-orth} it suffices to find a $u$-cluster
tilting module $X$ in the module category $\mod\,B_{N,n+1}$.  We
define $X$ to be the direct sum of the projective indecomposable
$B_{N,n+1}$-modules and the indecomposable modules $x_1, \ldots, x_n$
whose position in the stable AR quiver of $B_{N,n+1}$ is given by
Figure \ref{fig:Anx}.
\begin{figure}
\[
  \xymatrix @+1pc @!0 {
& & & & {} \ar[dr] & & x_n \ar[dr] & & *{\circ} \ar[dr] & \\
& & & & & x_{n-1} \ar[dr] \ar[ur] & & *{\circ} \ar[dr] \ar[ur] & & \ddots \\
& & \ddots \ar[dr] & & \adots \ar[ur] & & *{\circ} \ar[dr] \ar[ur] & & \ddots & \\
& \ddots \ar[dr] & & x_3 \ar[dr] \ar[ur] & & \adots \ar[ur] & & \ddots & & \\
\ddots \ar[dr] & & x_2 \ar[dr] \ar[ur] & & *{\circ} \ar[dr] \ar[ur] & & & & & \\
& x_1 \ar[ur] & & *{\circ} \ar[ur] & & {} & & & & \\
            }
\]
\caption{The indecomposable modules $x_1, \ldots, x_n$ for the
Nakayama algebra} 
\label{fig:Anx}
\end{figure}
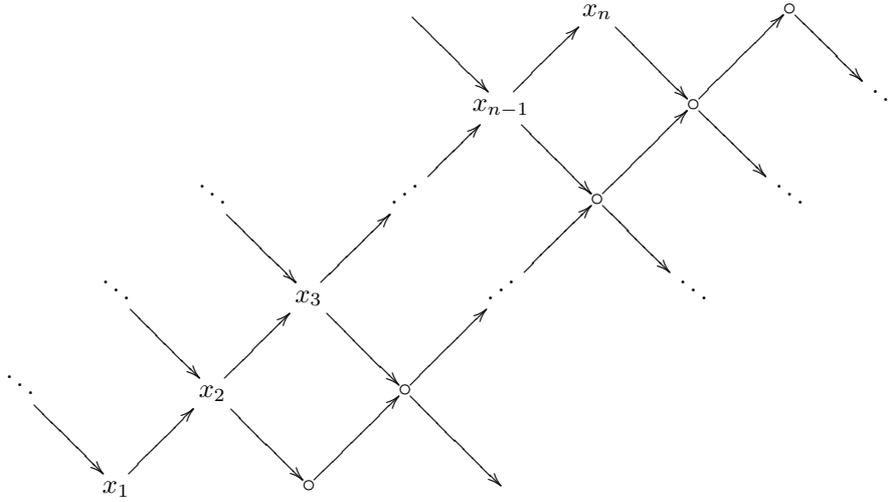
For the (uniserial) Nakayama algebras $B_{N,n+1}$ it is well-known
that the $i$th layer from the bottom of the stable AR quiver contains
precisely the non-projective indecomposable modules of dimension $i$
(see e.g. \cite[cor. V.4.2]{ASS}).  Moreover, the arrow from $x_i$ to
$x_{i+1}$ in the above picture is a monomorphism for each $i$.  From
this follows easily that the stable endomorphism ring of the module
$X$ is isomorphic to $kA_n$.

We now show that the module $X$ defined above is $u$-cluster tilting.
The $u$-cluster tilting modules (also called maximal $u$-orthogonal
modules) for selfinjective algebras of finite type with tree class
$A_n$ were described combinatorially in \cite[sec.\ 4]{Iyama}.  We
briefly sketch the main ingredients and refer to \cite{Iyama} for
details. On the stable AR quiver of $B_{N,n+1}$ one introduces a
coordinate system as in Figure \ref{fig:An_coordinates}.
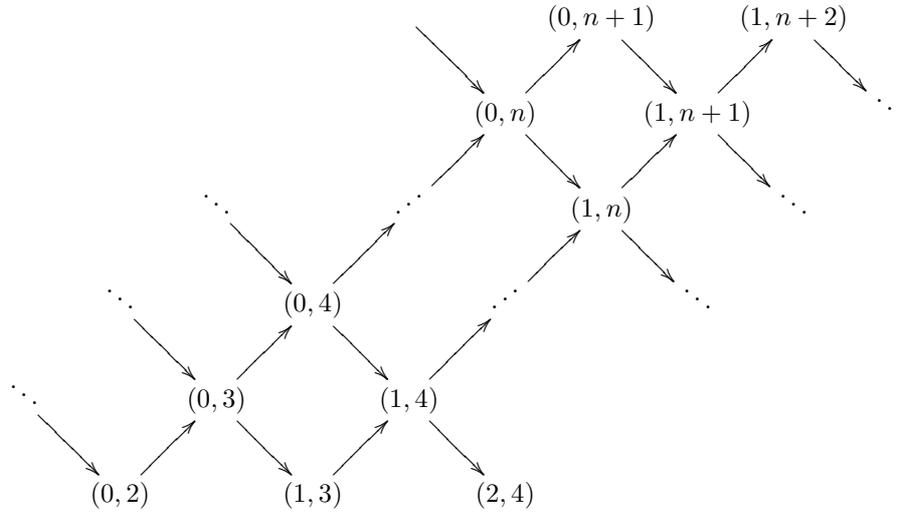
\begin{figure}
\[
  \vcenter{
  \xymatrix @+1pc @!0 {
& & & & {} \ar[dr] & & (0,n+1) \ar[dr] & & (1,n+2) \ar[dr] & \\
& & & & & (0,n) \ar[dr] \ar[ur] & & (1,n+1)\ar[dr] \ar[ur] & & \ddots \\
& & \ddots \ar[dr] & & \adots \ar[ur] & & (1,n) \ar[dr] \ar[ur] & & \ddots & \\
& \ddots \ar[dr] & & (0,4) \ar[dr] \ar[ur] & & \adots \ar[ur] & & \ddots & & \\
\ddots \ar[dr] & & (0,3) \ar[dr] \ar[ur] & & (1,4) \ar[dr] \ar[ur] & & & & & \\
& (0,2) \ar[ur] & & (1,3) \ar[ur] & & (2,4) & & & & \\
            }
          }
\]
\caption{The coordinate system for the Nakayama algebra}
\label{fig:An_coordinates}
\end{figure}
The first coordinate has to be taken modulo $N$.  To each vertex $x$
in the stable AR quiver one associates a `forbidden region' $H^{+}(x)$ which
is just the rectangle spanned from $x$ to the right; more precisely,
if $x=(i,j)$, then $H^{+}(x)$ is the rectangle with corners $x=(i,j)$,
$(i,i+n+1)$, $(j-2,i+n+1)$ and $(j-2,j)$ shown in Figure
\ref{fig:An_region}. 
\begin{figure}
\[
  \xymatrix @!0 @+2pc {
    \cdots \ar@{-}[rrrrr] & & & *{} \ar@{-}[dr] & & \cdots \\
    & & & H^+(x) & *{} & \\
    & x \ar@{-}[uurr] \ar@{-}[dr] & & & & \\
    \cdots \ar@{-}[rrrrr] & & *{} \ar@{-}[uurr] & & & \cdots \\
                      }
\]
\caption{The set $H^+(x)$ in Dynkin type $A$}
\label{fig:An_region}
\end{figure}
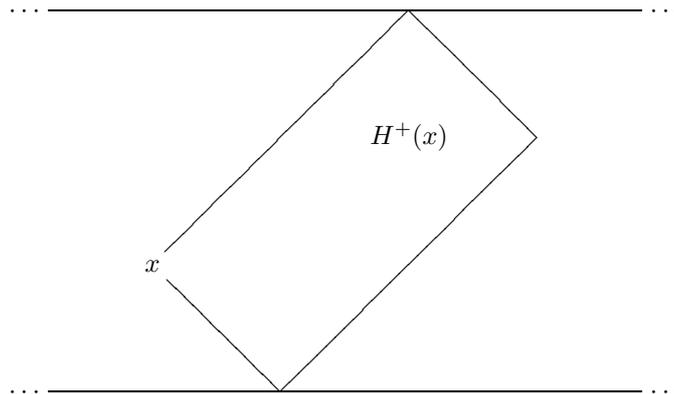
Define an automorphism $\omega$ on the stable AR quiver by setting
$\omega(i,j) = (j-n-2,i+1)$ and let $\tau$ be the usual AR
translation, $\tau(i,j)=(i-1,j-1)$.  Then a subset $S$ of the vertex
set $M$ in the stable AR quiver is called $u$-cluster tilting if
\[
  M\setminus S = \bigcup_{x\in S,\: 0<i\le u} H^+(\tau^{-1}\omega^{-i+1}x).
\]
For our particular choice of the $B_{N,n+1}$-module $X$ the set $S$ is
given by the above `slice' $x_1, \ldots, x_n$.  Then the
straightforward, but crucial, observation is that for $i=1,\ldots,u$
the sets $H(i) = \bigcup_{x\in S} H^+(\tau^{-1}\omega^{-i+1}x)$ are as
shown in Figure \ref{fig:An_regions}.
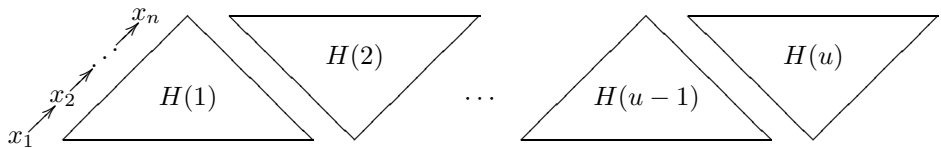
\begin{figure}
\[
  \xymatrix @-0.7pc @!0 {
    & & & *{x_n} & *{} \ar@{-}[dddrrr] & *{} \ar@{-}[dddrrr] \ar@{-}[rrrrrr] & & & & & &*{} & & & & *{} \ar@{-}[dddrrr] & *{} \ar@{-}[dddrrr] \ar@{-}[rrrrrr] & & & & & &*{}&*{}\\
    & & *{\adots} \ar[ur]& & & & & & H(2) & & & & & & & & & & & H(u) & & & *{} &\\
    & *{x_2} \ar[ur]& & & H(1) & & & & & & & \cdots & & & & H(u-1) & & & & & &*{} & &\\
   *{x_1} \ar[ur]& *{} \ar@{-}[uuurrr] \ar@{-}[rrrrrr] & & & & & & *{} & *{} \ar@{-}[uuurrr]& & & & *{} \ar@{-}[uuurrr] \ar@{-}[rrrrrr] & & & & & & *{} & *{} \ar@{-}[uuurrr]& *{}& &&\\
                        }
\]
\caption{The sets $H(i)$ for the Nakayama algebra}
\label{fig:An_regions}
\end{figure}
I.e., each $H(i)$ contains all the vertices in a triangular region
of the stable AR quiver with each edge of the triangle containing
$n$ vertices. 

Recall that $u$ is even by assumption.  In total, the union of the
forbidden regions $\bigcup_{0<i\le u}H(i)$ covers precisely the region
of the stable AR quiver between the slice $x_1, \ldots, x_n$ and the
shift of it by $\frac{u}{2}(n+1)+1$ units to the right.  But the
stable AR quiver has a circumference of $N = \frac{u}{2}(n+1)+1$
units, so it is clear from the above discussion that the set $S$
is $u$-cluster tilting and that, accordingly, the $B_{N,n+1}$-module
$X$ is indeed $u$-cluster tilting.

\medskip
{\em Vanishing of negative self-extensions. }
We must show $\underline{\Hom}(X,\Sigma^{-i} X) = 0$ for $i=1, \ldots,
u-1$.  For this, we can view $X$ in $\stab\, B_{N,n+1}$ where it has
the $n$ indecomposable summands $x_1, \ldots, x_n$.  Given
non-projective indecomposable $B_{N,n+1}$-modules $v$ and $w$, observe
that by \cite[sec.\ 4.2 and prop.\ 4.4.3]{Iyama}, we have
$\underline{\Hom}(v,w) = 0$ precisely if the vertex of $w$ is outside
the forbidden region $H^+(v)$.  So we need to check that all vertices
corresponding to indecomposable summands of $\Sigma^{-i} X$ for $i =
1, \ldots, u-1$ are outside the forbidden region $H(X) = \bigcup_j
H^+(x_j)$, where the union is over the indecomposable summands $x_j$
in $X$.

Now, the action of $\Sigma^{-1}$ on the stable AR quiver is just
$\omega$.  For instance, Figure \ref{fig:AnH} shows the forbidden
region along with the direct summands of $X$ and of $\omega X$.
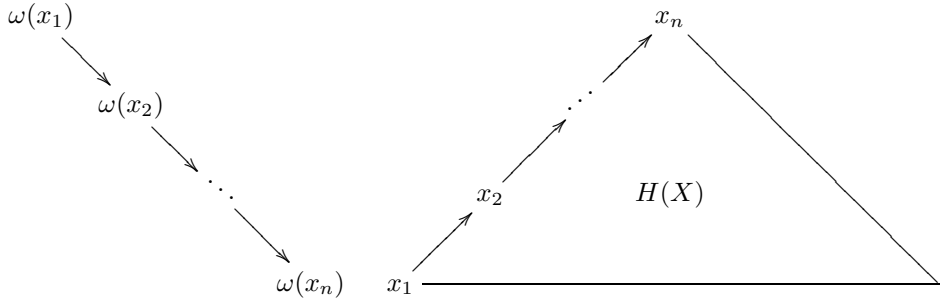
\begin{figure}
\[
  \vcenter{
  \xymatrix @+0.8pc @!0 {
    \omega(x_1) \ar[dr] & & & & & & & x_n \ar@{-}[dddrrr] \\
    & \omega(x_2) \ar[dr] & & & & & \adots \ar[ur] & & \\
    & & \ddots \ar[dr] & & & x_2 \ar[ur] & & H(X) \\
    & & & \omega(x_n) & x_1 \ar[ur] \ar@{-}[rrrrrr] & & & &&&\\
                      }
          }
\]
\caption{The set $H(X)$ and direct summands of $X$ and $\omega X$ for
the Nakayama algebra}
\label{fig:AnH}
\end{figure}
It is clear that the $\omega(x_j)$ fall outside $H(X)$.  More
generally, $\omega$ moves vertices to the left, so the only way we
could fail to get $\underline{\Hom}(X,\Sigma^{-i} X) = 0$ would be if
we took $i$ so large that the $\omega^i(x_j)$ made it all the way
around the stable AR quiver and reached the forbidden region from the right.
Let us check that this does not happen: $\omega^2$ is just a shift by
$n+1$ units to the left, and hence $\omega^{u-2} =
(\omega^2)^{\frac{u}{2} - 1}$ is a shift by
$(\frac{u}{2} - 1)(n+1) = N - (n + 2)$
units to the left.  Since the stable AR quiver has circumference $N$
it is clear that by applying $\omega^{u-1}$
we do not reach the forbidden region from the right.
\end{proof}

\subsection{M\"{o}bius algebras}
\label{sec:Mobius}

This subsection proves part (ii) of Theorem A from the introduction.

For integers $p,s \geq 1$, consider the M\"{o}bius algebra $M_{p,s}$.
Following the notation in \cite[app. A2.1.2]{Asashiba}, this is the
path algebra of the quiver shown in Figure \ref{fig:Moebius} modulo
the following relations.
\begin{itemize}

  \item[{(i)}] $\alpha_p^{i}\cdots\alpha_0^{i} = 
\beta_p^{i}\cdots\beta_0^{i}$ for each $i\in\{0,\ldots,s-1\}$.

\medskip

  \item[{(ii)}] $\beta_0^{i+1}\alpha_p^{i}=0$, 
$\alpha_0^{i+1}\beta_p^{i}=0$ for each $i\in\{0,\ldots,s-2\}$,\\[2mm]
$\alpha_0^{0}\alpha_p^{s-1}=0$,  
$\beta_0^{0}\beta_p^{s-1}=0$.

\medskip

  \item[{(iii)}] Paths of length $p+2$ are equal to zero.  

\end{itemize}

\begin{figure}
\[
  \xymatrix @+1.25pc {
    & & *{\circ} \ar[dl]_{\beta_p^{s-1}} & \cdots \ar[l]_{\beta_{p-1}^{s-1}}& *{}\ar@{.}[ddrr]& *{}& \\
    & *{\circ} \ar[dl]_{\beta_0^{0}} \ar[d]^>>>>>>>{\alpha_0^{0}} & *{\circ} \ar[l]^<<<<<{\alpha_p^{s-1}}& \cdots \ar[l]^{\alpha_{p-1}^{s-1}}& *{} \ar@{.}[dr]& &*{} \\
    *{\circ} \ar[d]_{\beta_1^{0}} & *{\circ} \ar[d]^{\alpha_1^{0}}& & & & *{} & *{}\\
    \vdots \ar[d]_{\beta_{p-1}^{0}} & \vdots \ar[d]^{\alpha_{p-1}^{0}}& & & & \vdots & \vdots \\
    *{\circ} \ar[dr]_{\beta_p^{0}} & *{\circ} \ar[d]^<<<<<<{\alpha_p^{0}}& & & & *{\circ} \ar[u]^{\alpha_1^{2}} & *{\circ} \ar[u]_{\beta_1^{2}}\\
    & *{\circ} \ar[r]^>>>>>>{\alpha_0^{1}} \ar[dr]_{\beta_0^{1}} & *{\circ} \ar[r]^{\alpha_1^{1}} & \cdots \ar[r]^{\alpha_{p-1}^{1}} & *{\circ} \ar[r]^<<<<<<<{\alpha_p^{1}} & *{\circ} \ar[u]^>>>>>>>{\alpha_0^{2}} \ar[ur]_{\beta_0^{2}} & \\
    & & *{\circ} \ar[r]_{\beta_1^{1}} & \cdots \ar[r]_{\beta_{p-1}^{1}}& *{\circ} \ar[ur]_{\beta_p^{1}}& & \\
           }
\]
\caption{Quiver for the M\"obius algebra}
\label{fig:Moebius}
\end{figure}
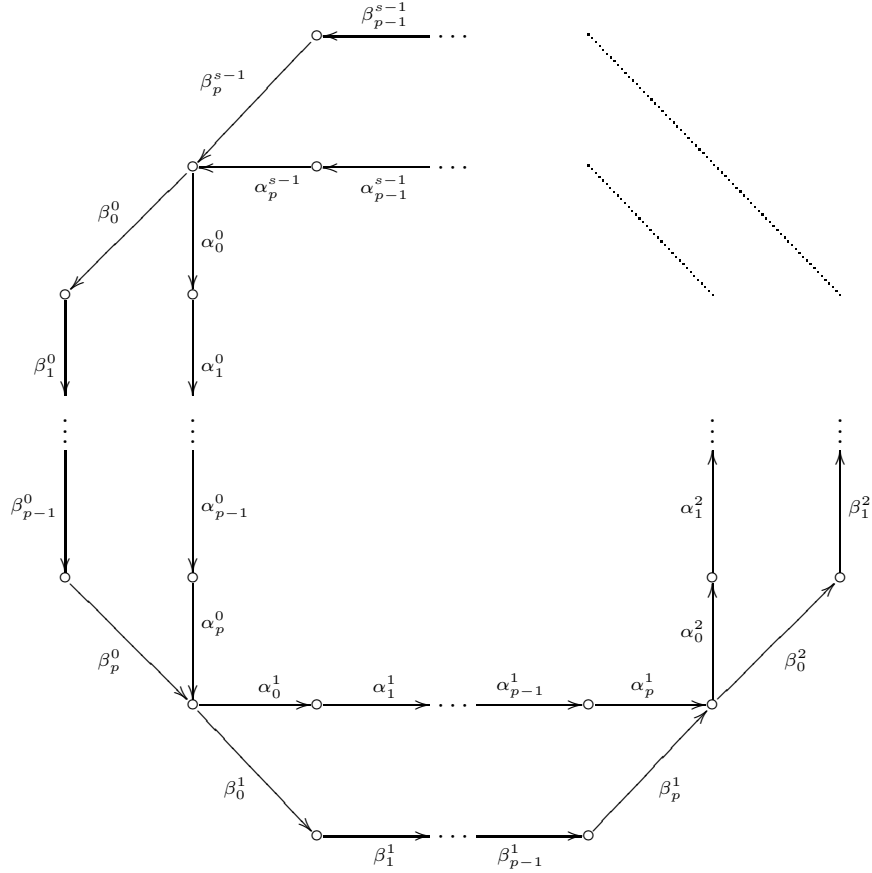

This is a selfinjective algebra of tree class $A_{2p+1}$.  In the
notation of \cite[app. A2.1.2]{Asashiba} the M\"obius algebra
$M_{p,s}$ is of the form $(A_{2p+1},s,2)$.  The stable AR quiver of
$M_{p,s}$ has the shape of a M\"{o}bius
band and can be obtained as $\BZ A_{2p+1}$ modulo a reflection in the
horizontal centre line composed with a shift by $s(2p+1)$ units to the
right, see \cite{Riedtmann}.

On the other hand, as we saw in Section \ref{sec:cluster}, if $u$ is
odd then the $u$-cluster category of type $A_{2p+1}$ has an AR quiver
which can be obtained as $\BZ A_{2p+1}$ modulo a reflection in the
horizontal centre line composed with a shift by
$\frac{u}{2}(2p+1+1) + 1 = u(p+1) + 1$
units to the right.  This quiver also has the shape of a M\"{o}bius
band, and again, this is no coincidence.

\begin{Theorem}
\label{thm:Mobius}
Let $u \geq 1$ be an odd integer and let $p,s \geq 1$ be integers for
which
\[
  s(2p+1) = u(p+1) + 1.
\]
Then the $u$-cluster category of type $A_{2p+1}$ is equivalent as a
triangulated category to the stable module category
$\stab\, M_{p,s}$.
\end{Theorem}

\begin{proof}
Like the proof of Theorem \ref{thm:Nakayama}, this proof is divided
into three sections verifying the conditions in Keller and Reiten's
Morita theorem \cite[thm.\ 4.2]{KellerReiten2}.

\medskip
{\em Calabi-Yau dimension. }
We must show that $\stab\, M_{p,s}$ has Calabi-Yau dimension $u + 1$.
Again this can be done using the work of Dugas, namely \cite[prop.\
9.6]{Dugas}.  There he shows that the Calabi-Yau dimension of the
stable module category $\stab\, M_{p,s}$ is of the form
$K_{p,s}(2p+1)-1$ where
\[
  K_{p,s} = \inf \big\{\,r\,\big|\,r\ge 1,\,r(p+1)\equiv 1 \,\modulo\, s,\,
  \mbox{and $\displaystyle \frac{r(s+p+1)-1}{s}$ is even} \big\}.
\]
Let us determine the number $K_{p,s}$ for the values of $u,p$ and $s$
given by the assumptions of the theorem.  We have 
\[
  u+2 = \frac{s(2p+1)-1}{p+1} +2 = \frac{(s+1)(2p+1)}{p+1}.
\]
Since $\gcd(p+1,2p+1)=1$, we deduce that $p+1$ divides $s+1$.
Moreover, the integer $\frac{s+1}{p+1}$ is odd since $u$ is odd by
assumption.  Now, for the condition $r(p+1)\equiv \, 1 \,\modulo\, s$,
the integer $\frac{s+1}{p+1}$ is clearly the minimal (positive)
solution.  Moreover, for this value $r=\frac{s+1}{p+1}$ we have that
\[
  \frac{r(s+p+1)-1}{s}
  = \frac{(s+1)(s+p+1)-(p+1)}{(p+1)s}
  = \frac{s+p+2}{p+1} = \frac{s+1}{p+1}+1
\]
is even. Hence $K_{p,s}=\frac{s+1}{p+1}$, and we conclude that
$\stab\,M_{p,s}$ has Calabi-Yau dimension
\begin{eqnarray*}
  K_{p,s}(2p+1)-1 & = & \frac{(s+1)(2p+1)}{p+1} - 1
  = \frac{(s+1)(2p+1)-2(p+1)}{p+1} + 1 \\
  & = & \frac{s(2p+1)-1}{p+1} + 1 = u + 1,
\end{eqnarray*}
where the last equality holds by assumption on $u$. 

\medskip
{\em $u$-cluster tilting object. }
To find a $u$-cluster tilting object $X$ in $\stab\,M_{p,s}$, recall
that the projective indecomposable $M_{p,s}$-modules are either
u\-ni\-se\-ri\-al or biserial, and that correspondingly, the vertices
in the quiver of $M_{p,s}$ are called uniserial or biserial. The
position of the corresponding simple modules in the stable AR quiver is
well-known; in particular, the simple modules corresponding to
biserial vertices occur in the centre line of the stable AR quiver.
As in the case of Nakayama algebras, we define the module $X$
as the direct sum of the projective indecomposable modules and the
indecomposable modules $x_1, \ldots, x_{2p+1}$ lying on a slice as in
Figure \ref{fig:A2pp1} such that the module $x_{p+1}$ is a simple
module $S_v$ corresponding to a biserial vertex $v$ of the quiver of
$M_{p,s}$.
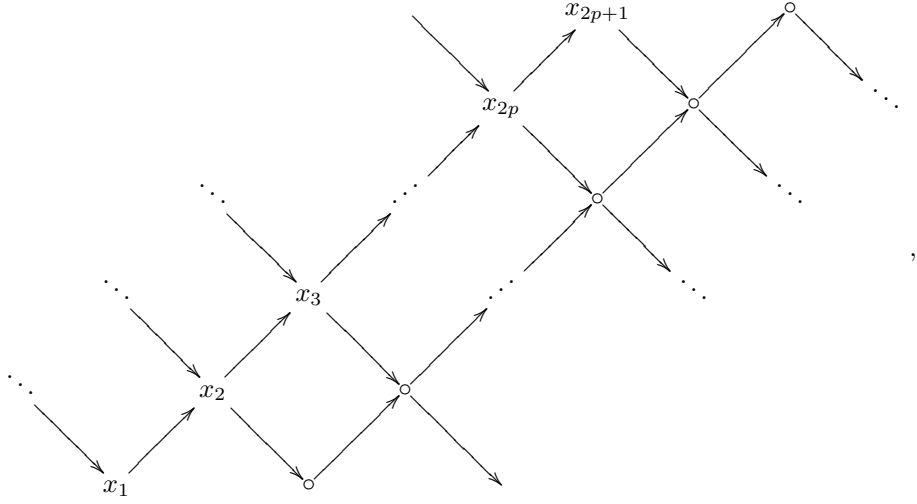
\begin{figure}
\[
  \vcenter{
  \xymatrix @+1pc @!0 {
& & & & {} \ar[dr] & & x_{2p+1} \ar[dr] & & *{\circ} \ar[dr] & \\
& & & & & x_{2p} \ar[dr] \ar[ur] & & *{\circ} \ar[dr] \ar[ur] & & \ddots \\
& & \ddots \ar[dr] & & \adots \ar[ur] & & *{\circ} \ar[dr] \ar[ur] & & \ddots & \\
& \ddots \ar[dr] & & x_3 \ar[dr] \ar[ur] & & \adots \ar[ur] & & \ddots & & \\
\ddots \ar[dr] & & x_2 \ar[dr] \ar[ur] & & *{\circ} \ar[dr] \ar[ur] & & & & & \\
& x_1 \ar[ur] & & *{\circ} \ar[ur] & & {} & & & & \\
            }
          },
\]
\caption{The indecomposable modules $x_1, \ldots, x_{2p+1}$ for the
M\"obius algebra} 
\label{fig:A2pp1}
\end{figure}
The other modules in this slice can also be described.  For
$j=1,\ldots,p$, the module $x_j$ is the uniserial module of length
$p+2-j$ with top $S_v$, and the module $x_{p+j+1}$ is the uniserial
module of length $j+1$ with socle $S_v$.

In particular, the bottom $p$ maps are epimorphisms and the upper $p$
maps are monomorphisms.  The composition of all $2p$ maps in such a
slice is non-zero, mapping the top onto the socle.  Most importantly
for us, it does not factor through a projective module, i.e., it is a
non-zero morphism in the stable module category of $M_{p,s}$.  From
this it follows easily that the stable endomorphism ring of the module
$X$ is isomorphic to $kA_{2p+1}$.

We now show that $X$ is $u$-cluster tilting. This argument is also
analogous to the Nakayama algebra case. The crucial difference is that
now $u$ is odd.  Hence, the forbidden regions defined in \cite{Iyama}
and discussed in the proof of Theorem \ref{thm:Nakayama} are as in
Figure \ref{fig:A2pp1_regions}.
\begin{figure}
\[
  \vcenter{
  \xymatrix @-0.55pc @!0 {
    & & *{{\scriptstyle x}_{\scriptscriptstyle 2p+1}} & *{} \ar@{-}[ddrr] & *{} \ar@{-}[ddrr] \ar@{-}[rrrr] & & & &*{} & & & & *{} \ar@{-}[ddrr] & *{} \ar@{-}[ddrr] \ar@{-}[rrrr] & & & & *{} &*{}\ar@{-}[ddrr]&*{*}&&\\
    & *{\adots} \ar[ur] & &{\scriptstyle H(1)} & & & {\scriptstyle H(2)} & & & \cdots & & &{\scriptstyle H(u-2)} & & & {\scriptstyle H(u-1)} & & &{\scriptstyle H(u)} & & *{} &\\
    *{{\scriptstyle x}_{\scriptscriptstyle 1}} \ar[ur] & *{} \ar@{-}[uurr] \ar@{-}[rrrr] & & & & *{} & *{} \ar@{-}[uurr] & & && *{} \ar@{-}[uurr] \ar@{-}[rrrr] & & & & *{} & *{} \ar@{-}[uurr] & *{}\ar@{-}[uurr]\ar@{-}[rrrr]&&&&*{}&*{} \\
                         }
          }
\]
\caption{The sets $H(i)$ for the M\"obius algebra}
\label{fig:A2pp1_regions}
\end{figure}
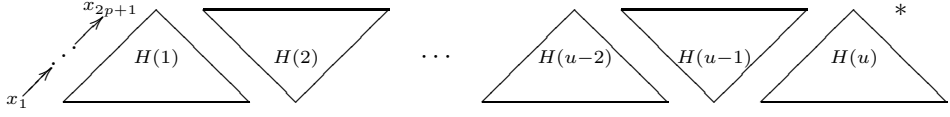
The method from the proof of Theorem \ref{thm:Nakayama} shows that, to
see that $X$ is $u$-cluster tilting, it is sufficient to see that
the vertex $x_1$ is identified with the vertex $*$.
However, each $H(i)$ contains the vertices of the stable AR quiver in an
equilateral triangular region with edges having $2p+1$ vertices.  So
in order for $x_1$ to be identified with $*$, we must identify after
$\frac{u-1}{2}(2p+2)+ (p+2) = u(p+1)+1$
units.  But in fact, one gets the stable AR quiver of $M_{p,s}$ from
$\mathbb{Z}A_{2p+1}$ by identifying after $s(2p+1)$ units, and by the
assumption of the theorem we do indeed have $s(2p+1)=u(p+1)+1$. 

\medskip
{\em Vanishing of negative self-extensions. }
We must show $\underline{\Hom}(X,\Sigma^{-i} X) = 0$ for $i=1, \ldots,
u-1$.  The proof is analogous to the Nakayama case: The action of
$\Sigma^{-1}$ on the stable AR quiver is again just $\omega$, and the
forbidden region of $X$ along with the direct summands of $X$ and of
$\omega X$ are as in Figure \ref{fig:AnH2}.
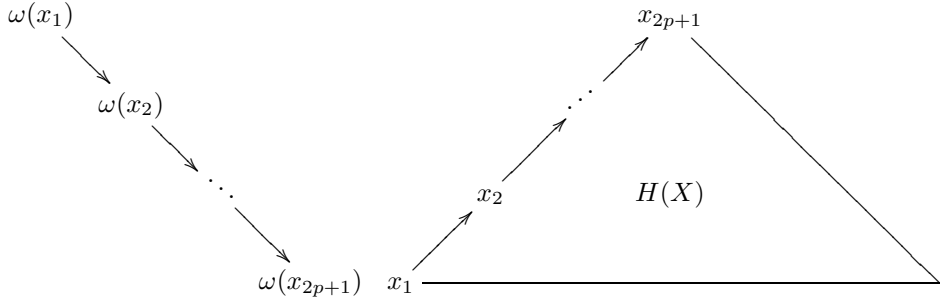
\begin{figure}
\[
  \xymatrix @+0.8pc @!0 {
    \omega(x_1) \ar[dr] & & & & & & & x_{2p+1} \ar@{-}[dddrrr] & \\
    & \omega(x_2) \ar[dr] & & & & & \adots \ar[ur] & & & \\
    & & \ddots \ar[dr] & & & x_2 \ar[ur] & & H(X) & \\
    & & & \omega(x_{2p+1}) & x_1 \ar[ur] \ar@{-}[rrrrrr] & & & & & & *{} & \\
                        }
\]
\caption{The set $H(X)$ and direct summands of $X$ and $\omega X$ for
the M\"obius algebra}
\label{fig:AnH2}
\end{figure}
The only way we could fail to get $\underline{\Hom}(X,\Sigma^{-i} X) =
0$ would be if we took $i$ so large that the $\omega^i(x_j)$ made it
all the way around the stable AR quiver and reached the forbidden region from
the right.  In fact, let us look at the largest relevant integer,
$u - 1$.  As $\omega^2$ is just a shift by $2p+2$ units to the left, we
have that $\omega^{u-1} = (\omega^2)^{\frac{u-1}{2}}$ is a shift by
$\frac{u-1}{2}(2p+2) = (u-1)(p+1) = s(2p+1) - (p+2)$
units to the left.  The stable AR quiver has a circumference of $s(2p+1)$
units, so the $\omega^{u-1}(x_j)$ lie strictly to the right of the
forbidden region.  (Note that the stable AR quiver is a M\"{o}bius band, and
the change of orientation means that, although $u-1$ is even, the
$\omega^{u-1}(x_j)$ form a diagonal line perpendicular, not parallel,
to the line of the $x_j$.)  So $\underline{\Hom}(X,\Sigma^{-i} X)$ is
zero for $i = u - 1$, and hence certainly also for all values $i = 1,
\ldots, u - 1$.  This completes the proof.
\end{proof}

\begin{Remark}
Note that as a special case of Theorem \ref{thm:Mobius}, the
$1$-cluster category of type $A_3$ is triangulated equivalent to
$\stab\, M_{1,1}$.  The M\"obius algebra $M_{1,1}$ is isomorphic to
the preprojective algebra of Dynkin type $A_3$.

This is the only case where a $1$-cluster category is triangulated
equivalent to the stable module category of a selfinjective algebra of
finite representation type and tree class $A_n$.  This follows from
the complete classification of representation-finite selfinjective
algebras of stable Calabi-Yau dimension 2 given in \cite[cor.
3.10]{ES}.
\end{Remark}

\section{Dynkin type $D$}
\label{sec:typeD}

This section proves Theorem D from the introduction.

Asashiba's paper \cite{Asashiba1} gives a derived and stable
equivalence classification of selfinjective algebras of finite
representation type.  If the tree class of the stable AR quiver is
Dynkin type $D$, then there are three families of representatives of
algebras denoted
\begin{itemize}

  \item  $(D_n,s,1)$ with $n \geq 4$, $s \geq 1$,

\medskip
  
  \item  $(D_n,s,2)$ with $n \geq 4$, $s \geq 1$,

\medskip

  \item  $(D_{3m},\frac{s}{3},1)$ with $m \ge 2$, $s \geq 1$, $3 \nmid s$.

\end{itemize}
It follows from \cite[cor.\ 1.7]{BS} that the stable AR quivers of
these algebras are cylinders with the following circumferences.
\begin{itemize}

  \item  For $(D_n,s,1)$ and $(D_n,s,2)$ the circumference is $s(2n-3)$.

\medskip

  \item  For $(D_{3m},\frac{s}{3},1)$ the circumference is $s(2m-1)$.

\end{itemize}

By Subsection \ref{subsec:clusterD}, the AR quiver of the $u$-cluster
category of type $D_n$ is a cylinder of circumference $u(n-1)+1$.  So in
order for the stable categories $\stab\,(D_n,s,1)$ or
$\stab\,(D_n,s,2)$ to be $u$-cluster categories we need
\[
  u(n-1)+1 = s(2n-3).
\]
In particular, this implies
\[
  u\equiv\,-(n-1)^{-1}\,\equiv\,-2 \,\modulo\, (2n-3).
\]
Likewise, for the stable category $\stab\,(D_{3m},\frac{s}{3},1)$ to
be a $u$-cluster category, we need
\begin{equation}
\label{equ:d}
  u(3m-1)+1 = s(2m-1).
\end{equation}
In particular, this implies
\[
  u\equiv\,-m^{-1}\,\equiv\,-2 \,\modulo\, (2m-1).
\]
Moreover, recall that in the definition of the algebras
$(D_{3m},\frac{s}{3},1)$ the case $3 \mid s$ is excluded. 
In the situation of equation \eqref{equ:d} we have 
\[
  3 \nmid s 
  \; \Longleftrightarrow \;
  u(3m-1)+1 \not\equiv \, 0 \,\modulo\, 3(2m-1)
  \; \Longleftrightarrow \;
  u \not\equiv \, -(3m-1)^{-1} \, \equiv \, -2 \,\modulo\, (6m-3).
\]

Indeed, these conditions turn out also to be sufficient.  Note that,
setting $n=3m$, the forbidden case $u \equiv -2 \,\modulo\, (6m-3)$ for
the algebras $(D_{3m},\frac{s}{3},1)$ is precisely the case
$u\,\equiv\,-2 \,\modulo\, (2n-3)$ in which the algebras $(D_n,s,1)$ and
$(D_n,s,2)$ can be applied.

The main result of this section is the following which restates
Theorem D from the introduction.

\begin{Theorem}
\label{thm:n_even}
Let $m, n, u$ be integers with $u \geq 1$.
\begin{enumerate}

\item Suppose that $n\ge 4$ is even and $u \equiv -2 \,\modulo\,
  (2n-3)$.  Then the $u$-cluster category of type $D_{n}$ is
  equivalent as a triangulated category to the stable module category
\[
  \stab\, (D_{n},\frac{u(n-1)+1}{2n-3},1).
\]

\medskip

\item
Suppose that $n\ge 5$ is odd and $u \equiv -2 \,\modulo\, (2n-3)$.

\medskip
\noindent
If $u$ is even, then the $u$-cluster category
of type $D_{n}$ is triangulated e\-qui\-va\-lent to the stable 
module category
$\stab\, (D_{n},\frac{u(n-1)+1}{2n-3},1).$

\medskip
\noindent
If $u$ is odd, then the $u$-cluster category
of type $D_{n}$ is triangulated e\-qui\-va\-lent to the stable 
module category
$ \stab\, (D_{n},\frac{u(n-1)+1}{2n-3},2).$

\medskip

\item
Suppose that $m\ge 2$ and $u\equiv\, -2 \,\modulo\, (2m-1)$ but 
$u\not\equiv\, -2 \,\modulo\, (6m-3)$. 
Suppose moreover that not both $m$ and $u$ are odd.
Then the $u$-cluster category
of type $D_{3m}$ is equivalent as a triangulated category to the stable 
module category $\stab\, (D_{3m},\frac{s}{3},1)$
where $s=\frac{u(3m-1)+1}{2m-1}$. 

\end{enumerate}
\end{Theorem}

\begin{proof}
As in type A, the proof is divided into three sections verifying the
conditions in Keller and Reiten's Morita theorem \cite[thm.\
4.2]{KellerReiten2}.

\medskip
{\em Calabi-Yau dimension. }
We must show that each of the stable module categories occurring in
the theorem has Calabi-Yau dimension $u+1$.

For part (i) we suppose that $n\ge 4$ is even and we consider the
algebra $(D_{n},\frac{u(n-1)+1}{2n-3},1)$. The Calabi-Yau dimension of
its stable module category can be determined using \cite[thm.\
6.1]{Dugas}, in which both parts can apply. The relevant invariants
occurring there are the frequency $f=\frac{u(n-1)+1}{2n-3}$, the
Coxeter number $h_{D_n}=2n-2$ and the related number $h_{D_n}^*=
h_{D_n}/2=n-1$, and $m_{D_n}=h_{D_n}-1=2n-3$.

If \cite[thm.\ 6.1(1)]{Dugas} applies then the Calabi-Yau
dimension $d$ of $\stab\, (D_{n},\frac{u(n-1)+1}{2n-3},1)$ satisfies
$$d \equiv 1 - (h_{D_n}^*)^{-1} \,\modulo\, fm_{D_n}
\equiv 1- (n-1)^{-1} \,\modulo\, (u(n-1)+1)
$$
and $0 < d \le u(n-1)+1$. Upon multiplication with $n-1$ this becomes
$d(n-1) \equiv n-2 \,\modulo\, (u(n-1)+1)$
which is easily checked to be satisfied by $d=u+1$. 

If \cite[thm.\ 6.1(2)]{Dugas} applies then the Calabi-Yau
dimension $d$ of $\stab\, (D_{n},\frac{u(n-1)+1}{2n-3},1)$ has the
form $d=2r+1$ where $r$ is determined by
\begin{equation}
\label{equ:Dn:neven}
r \equiv  - (h_{D_n})^{-1} \,\modulo\, fm_{D_n}
\equiv -(2n-2)^{-1} \,\modulo\, (u(n-1)+1)
\end{equation}
and $0 \le r < u(n-1)+1$. Since \cite[thm.\ 6.1(2)]{Dugas}
applies we know from the assumptions stated in \cite[thm.\
6.1(1)]{Dugas} that $2\nmid f=\frac{u(n-1)+1}{2n-3}$ from which it
follows that $u$ is even (since $n$ is even). Setting $r=\frac{u}{2}$
it is readily checked that it satisfies (\ref{equ:Dn:neven}).
Therefore the Calabi-Yau dimension is $2r+1=u+1$, as required.

For part (ii), we suppose that $n\ge 5$ is odd and we consider the
algebras $(D_{n},\frac{u(n-1)+1}{2n-3},1)$ and
$(D_{n},\frac{u(n-1)+1}{2n-3},2)$, depending on whether $u$ is even or
odd.

If $u$ is even then the Calabi-Yau dimension of $\stab\,
(D_{n},\frac{u(n-1)+1}{2n-3},1)$ can be determined using
\cite[thm.\ 6.1(2)]{Dugas} (note that \cite[thm.\ 6.1(1)]{Dugas}
only applies for $n$ even).  The only difference to the case of $n$
even is the invariant $h_{D_n}^*$ which is now equal to $2n-2$ instead
of $n-1$. But this invariant does not occur in \cite[thm.\
6.1(2)]{Dugas} so the proof for $n$ even carries over verbatim and gives
that $\stab\, (D_{n},\frac{u(n-1)+1}{2n-3},1)$ has Calabi-Yau
dimension $u + 1$.

If $u$ is odd (and $n \geq 5$ is still odd) then the Calabi-Yau
dimension of $\stab\,(D_{n},\frac{u(n-1)+1}{2n-3},2)$ can be
determined using \cite[prop.\ 7.3]{Dugas}. Note that since $n$ is odd,
the frequency $f=\frac{u(n-1)+1}{2n-3}$ is odd as well, and hence
\cite[prop.\ 7.3(1)]{Dugas} applies. From this we get that the
Calabi-Yau dimension of $ \stab\, (D_{n},\frac{u(n-1)+1}{2n-3},2)$ is
of the form $d=2r$ where $r\equiv (n-2)(2n-2)^{-1} \,\modulo\,
(u(n-1)+1)$ and $0 < r < u(n-1)+1$. Upon multiplication with $2n-2$
the latter equation becomes $2r(n-1)\equiv n-2 \,\modulo\, (u(n-1)+1)$
which is easily seen to be satisfied by $r=\frac{u+1}{2}$. Therefore,
the Calabi-Yau dimension is $d=2r=u+1$, as required.

For part (iii), we consider the algebras $(D_{3m},\frac{s}{3},1)$
where $s=\frac{u(3m-1)+1}{2m-1}$.  The Calabi-Yau dimension of the
stable module category can again be determined using \cite[thm.\
6.1]{Dugas}.

If $m$ is even then the invariants we need are the frequency
$f=\frac{s}{3}=\frac{u(3m-1)+1}{3(2m-1)}$, the Coxeter number
$h_{D_{3m}}=6m-2$ and the related numbers $m_{D_{3m}}=h_{D_{3m}}-1=
6m-3$, and $h_{D_{3m}}^* = h_{D_{3m}}/2 = 3m-1$.

If \cite[thm.\ 6.1(1)]{Dugas} applies then the Calabi-Yau dimension
$d$ of $\stab\, (D_{3m},\frac{s}{3},1)$ is determined by
$$d \equiv 1 - (h_{D_{3m}}^*)^{-1} \,\modulo\, fm_{D_{3m}}
\equiv 1 - (3m-1)^{-1} \,\modulo\, (u(3m-1)+1)
$$ 
and $0 < d \le u(3m-1)+1$. Clearly, $d=u+1$ satisfies these properties
and hence the Calabi-Yau dimension is $u+1$, as claimed.

If \cite[thm.\ 6.1(2)]{Dugas} applies then the Calabi-Yau dimension
$d$ of $\stab\, (D_{3m},\frac{s}{3},1)$ is of the form $d=2r+1$ where
$r$ is determined by
\[
  r \equiv - (h_{D_{3m}})^{-1} \,\modulo\, fm_{D_{3m}}
    \equiv -(6m-2)^{-1} \,\modulo\, (u(3m-1)+1)
\]
and $0 \le r < u(3m-1)+1$.  Note that our assumptions in this case
imply that $u$ is even; otherwise the frequency $f$ would be even and
we would be in the situation of \cite[thm.\ 6.1(1)]{Dugas}. Setting $r
= \frac{u}{2}$ is easily seen to satisfy the above properties,
i.e. the Calabi-Yau dimension is $2r+1=u+1$, as desired.

Finally, if $m$ is odd then only \cite[thm.\ 6.1(2)]{Dugas} can
apply. The computation of the Calabi-Yau dimension carries over
verbatim from the previous one; in fact, by assumption in part (iii)
of our theorem $u$ has to be even (since $m$ is odd).  Hence also for
$m$ odd and $u$ even we get the Calabi-Yau dimension of $\stab\,
(D_{3m},\frac{s}{3},1)$ to be $u+1$, as required.

\medskip
{\em $u$-cluster tilting object. }
To find a $u$-cluster tilting object $X$ in the stable module
category, the method is the same for parts (i)--(iii).  In part (iii)
we set $n = 3m$ so that in each case $n$ denotes the number of
vertices in the underlying Dynkin quiver of type $D_n$.

Let $X$ be the sum of the projective indecomposable modules and the
indecomposable modules $x_1, \ldots, x_{n-2}, x_{n-1}^-, x_{n-1}^+$
whose positions in the stable AR quiver of the relevant algebra are
given by Figure \ref{fig:ZDn2}.
\begin{figure}
\[
  \def\objectstyle{\scriptstyle}
  \xymatrix @+1pc @!0 {
& & & & {} \ar[dr] & & {\scriptstyle x_{n-1}^+} \ar[dr] & & *{\circ} \ar[dr] & \\
& & & & \cdots \ar[r] & {\scriptstyle x_{n-2}} \ar[dr] \ar[ur] \ar[r] & {\scriptstyle x_{n-1}^-} \ar[r] & *{\circ} \ar[dr] \ar[ur] \ar[r] & *{\circ} \ar[r] & \ddots \\
& & \ddots \ar[dr] & & \adots \ar[ur] & & *{\circ} \ar[dr] \ar[ur] & & \ddots & \\
& \ddots \ar[dr] & & {\scriptstyle x_3} \ar[dr] \ar[ur] & & \adots \ar[ur] & & \ddots & & \\
\ddots \ar[dr] & & {\scriptstyle x_2} \ar[dr] \ar[ur] & & *{\circ} \ar[dr] \ar[ur] & & & & & \\
& {\scriptstyle x_1} \ar[ur] & & *{\circ} \ar[ur] & & {} & & & & \\
            }
\]
\caption{The indecomposable modules $x_1, \ldots, x_{n-2}, x_{n-1}^-,
x_{n-1}^+$ in Dynkin type $D$}
\label{fig:ZDn2}
\end{figure}
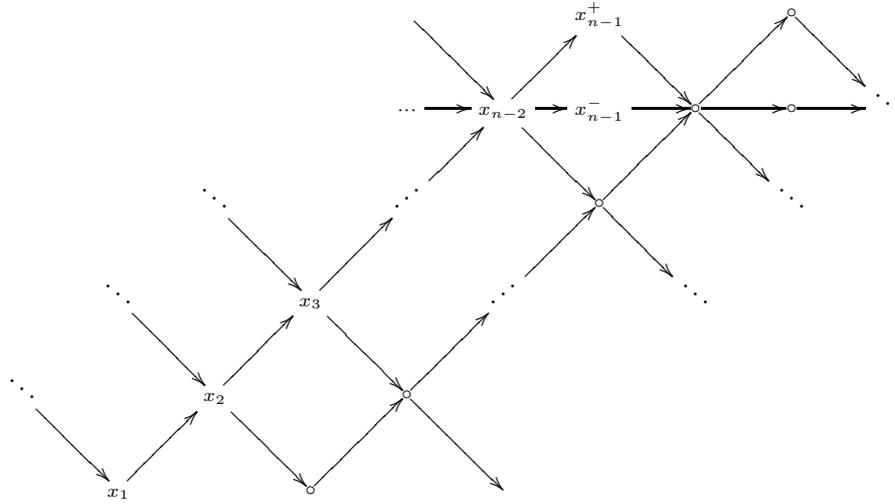
We show that $X$ is $u$-cluster tilting
in the stable module
category.  By Proposition \ref{prop:max-orth}, it is enough to prove
that it is 
$u$-cluster tilting in the abelian category of modules.  Following
\cite[def. 4.2]{Iyama}, introduce a coordinate system on the stable AR quiver
as in Figure \ref{fig:Dn_coordinates}.
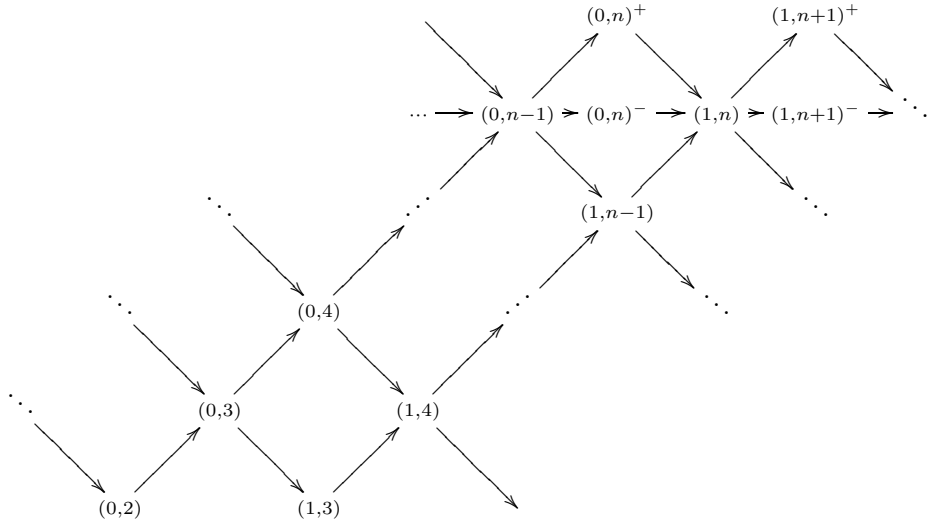
\begin{figure}
\[
  \def\objectstyle{\scriptstyle}
  \xymatrix @+1.1pc @!0 {
& & & & {} \ar[dr] & & (0,n)^+ \ar[dr] & & (1,n+1)^+ \ar[dr] & \\
& & & & \cdots \ar[r] & (0,n-1) \ar[dr] \ar[ur] \ar[r] & (0,n)^- \ar[r] & (1,n) \ar[dr] \ar[ur] \ar[r] & (1,n+1)^- \ar[r] & \ddots \\
& & \ddots \ar[dr] & & \adots \ar[ur] & & (1,n-1) \ar[dr] \ar[ur] & & \ddots & \\
& \ddots \ar[dr] & & (0,4) \ar[dr] \ar[ur] & & \adots \ar[ur] & & \ddots & & \\
\ddots \ar[dr] & & (0,3) \ar[dr] \ar[ur] & & (1,4) \ar[dr] \ar[ur] & & & & & \\
& (0,2) \ar[ur] & & (1,3) \ar[ur] & & {} & & & & \\
                        }
\]
\caption{The coordinate system in Dynkin type $D$}
\label{fig:Dn_coordinates}
\end{figure}
To each vertex $x$ in the stable AR quiver, associate a `forbidden
region' $H^+(x)$ defined as in Figure \ref{fig:DnH} (see
\cite[sec. 4.2]{Iyama}), with the proviso that if $x$ is not one of
the `exceptional' vertices indicated by superscripts $+$ and $-$, then
$H^+(x)$ contains all the exceptional vertices along the relevant part
of the top line in the diagram, but if $x$ is exceptional, say $x =
(i,i+n)^+$, then $H^+(x)$ only contains half the exceptional vertices
along the relevant part of the top line, namely $(i,i+n)^+,
(i+1,i+n+1)^-, (i+2,i+n+2)^+, \ldots$, starting with $x$ itself.
\begin{figure}
\[
  \xymatrix @+1pc @!0 {
& *{(i,i+n)} \ar@{-}[rrrr] & & & & *{(j-2,j+n-2)} \ar@{-}[dr] & \\
*{x = (i,j)} \ar@{-}[ur] \ar@{-}[ddrr] & & & {\textstyle H^+(x)} & & &*{(i+n-2,j+n-2) {\textstyle ,}} \\
& & & *{(j-2,i+n)} \ar@{-}[dr] & & & \\
& & *{(j-2,j)} \ar@{-}[ur] & & *{(i+n-2,i+n)} \ar@{-}[uurr] & & \\
            }
\]
\caption{The set $H^+(x)$ in Dynkin type $D$}
\label{fig:DnH}
\end{figure}
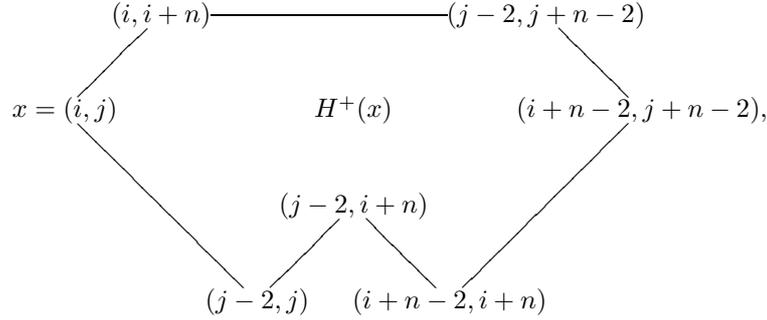
Define the following automorphisms of the stable AR quiver: $\theta$ is the
identity on the non-exceptional vertices and switches $(i,i+n)^+$ and
$(i,i+n)^-$.  The AR translation $\tau$ is given by moving each vertex
one unit to the left.  And finally, $\omega = \theta(\tau
\theta)^{n-1}$.  A subset $S$ of the vertex set $M$ in the stable AR quiver
is called 
$u$-cluster tilting if
\[
  M \setminus S = \bigcup_{x \in S, 0 < i \leq u} 
                    H^+(\tau^{-1}\omega^{-i+1}x),
\]
see \cite[sec. 4.2]{Iyama}.
For our choice of $X$, the set $S$ is given by the modules $x_1$,
$\ldots$, $x_{n-2}$, $x_{n-1}^-$, $x_{n-1}^+$.  But then the sets
\[
  H(i) = \bigcup_{x \in S} H^+(\tau^{-1}\omega^{-i+1}x)
\]
can easily be verified to sit in the stable AR quiver as in Figure
\ref{fig:DnHs} where each parallellogram has $n-1$ vertices on each
edge.
\begin{figure}
\[
  \def\objectstyle{\scriptstyle}
  \vcenter{
  \xymatrix @+0.65pc @!0 {
& & & *{x_{n-1}^+} & & *{} \ar@{-}[rr] & & *{} & & *{} \ar@{-}[rr] & & *{}\\
& & *{x_{n-2}} \ar[ur] \ar[r] & x_{n-1}^- & & {\textstyle H(1)} & & {\textstyle \cdots} & & {\textstyle H(u)} \\
& \adots \ar[ur] \\
*{x_1} \ar[ur] & & *{} \ar@{-}[uuurrr] \ar@{-}[rr] & & *{} \ar@{-}[uuurrr] & & *{} \ar@{-}[uuurrr] \ar@{-}[rr] & & *{}\ar@{-}[uuurrr]\\
                         }
          }
\]
\caption{The sets $H(i)$ in Dynkin type $D$}
\label{fig:DnHs}
\end{figure}
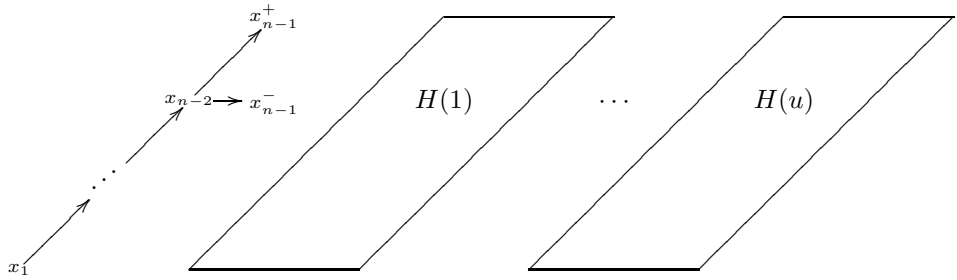
In total, the union
\[
  \bigcup_{0 < i \leq u} H(i) = 
  \bigcup_{x \in S, 0 < i \leq u} 
                    H^+(\tau^{-1}\omega^{-i+1}x)
\]
is a parallellogram with $u(n-1)$ vertices on each horizontal edge.
This means that the parallellogram covers precisely the region between
the $x$'s and their shift by $u(n-1) + 1$ units to the right.

By Subsection \ref{subsec:clusterD}, this is exactly the
number of units after which $\BZ D_n$ is identified with itself to get
the stable AR quiver.  It follows that $S$ is a 
$u$-cluster tilting
set of vertices of the stable AR quiver, and hence $X$ is 
$u$-cluster tilting 
in the module category by \cite[thm.\ 4.2.2]{Iyama}.

To show that the stable endomorphism algebra $\underline{\End}(X)$ is
$kD_n$, we need to see that for each pair of in\-de\-com\-po\-sa\-ble
summands $x_i$ and $x_j$ of $X$, the stable $\Hom$-space
$\underline{\Hom}(x_i,x_j)$ is one-dimensional if $x_i$ is below $x_j$
in the stable AR quiver, and zero otherwise.

The self-injective algebras in the theorem in question are standard,
so each morphism between in\-de\-com\-po\-sa\-ble modules in the
stable category is a sum of compositions of sequences of irreducible
morphisms between indecomposable modules.  Consider such a sequence
which composes to a morphism $x_i \rightarrow x_j$.

If, along the sequence, there is an indecomposable $y$ which is not a
summand of $X$, then $y \rightarrow x_j$ factors through a direct sum
of indecomposable summands of $\tau X$ by \cite[lem.\ VIII.5.4]{ASS}.
But then $x_i \rightarrow y \rightarrow x_j$ factors in the same way,
and this means that it is zero because $\underline{\Hom}(X,\tau X) =
0$ by the methods used in the proof that $X$ is 
$u$-cluster tilting.
Hence $x_i \rightarrow x_j$ can be taken to be a sum of compositions
of sequences of irreducible morphisms which only pass through
indecomposable summands of $X$.  In the stable AR quiver, the arrows between
these summands all point upwards, so it follows that
$\underline{\Hom}(x_i,x_j)$ is zero unless $x_i$ is below $x_j$ in the
stable AR quiver.

On the other hand, if $x_i$ is below $x_j$, then
$\underline{\Hom}(x_i,x_j)$ is non-zero by \cite[sec.\ 4.2 and prop.\ 
4.4.3]{Iyama}.  Finally, 
it follows from
\cite[satz 3.5]{Riedtmann2} that the dimension of
$\underline{\Hom}(x_i,x_j)$ is at most one.

\medskip
{\em Vanishing of negative self-extensions. }
We must show $\underline{\Hom}(X,\Sigma^{-i} X) = 0$ for $i=1, \ldots,
u-1$.  If $v$ and $w$ are indecomposable non-projective modules, we
have $\underline{\Hom}(v,w) = 0$ precisely if the vertex of $w$ is
outside the region $H^+(v)$, see \cite[sec.\ 4.2 and prop.\
4.4.3]{Iyama}.  So we need to check that all vertices corresponding to
indecomposable summands of $\Sigma^{-i}X$ for $i = 1, \ldots, u-1$ are
outside the forbidden region $H(X) = \bigcup_{x} H^+(x)$, where the
union is over the indecomposable summands of $X$.

But the action of $\Sigma^{-1}$ on the stable AR quiver is just
$\omega$.  So Figure \ref{fig:DnH2} shows the forbidden region along
with the $\Sigma^{-i}X$.
\begin{figure}
\[
  \def\objectstyle{\scriptstyle}
  \xymatrix @+0.45pc @!0 {
& & & {\textstyle \Sigma^{-(u-1)}X} & & {\textstyle \cdots} & & {\textstyle 
\Sigma^{-1}X} & & & {\textstyle X} \\
& & & *{\circ} & & & & *{\circ} & & & *{x_{n-1}^+} \ar@{-}[rr] & & *{} \\
& & *{\circ} \ar[ur] \ar[r] & *{\circ} & {\textstyle \cdots} & & *{\circ} \ar[ur] \ar[r] & *{\circ} & {\textstyle \cdots} & *{x_{n-2}} \ar[ur] \ar[r] & x_{n-1}^- \\
& \adots \ar[ur] & & & & \adots \ar[ur] & & & \adots \ar[ur] & {\textstyle H(X)} \\
*{\circ} \ar[ur] & & & & *{\circ} \ar[ur] & & & *{x_1} \ar[ur] \ar@{-}[rr] & & *{} \ar@{-}[uuurrr] \\
                         }
\]
\caption{The set $H(X)$ and direct summands of $\Sigma^{-(u-1)}X,
\ldots, \Sigma^{-1}X, X$ in Dynkin type $D$}
\label{fig:DnH2}
\end{figure}
\label{page:D2}
The only way we could fail to get $\underline{\Hom}(X,\Sigma^{-i} X) =
0$ would be if we took $i$ so large that $\Sigma^i X$ made it all the
way around the stable AR quiver and reached the forbidden region from the
right.

However, this does not happen: $\omega$, and hence $\Sigma^{-1}$,
is a move by $n-1$ units to the left, so $\Sigma^{-(u-1)}X$ is
moved 
$(u-1)(n-1)$
units to the left.  On the other hand, to reach $H(X)$, one has
to move by the circumference of the stable AR quiver minus the
horizontal length of $H(X)$ plus one, and this is
$u(n-1) + 1 - (n - 1) + 1 = (u-1)(n-1) + 2$.
\end{proof}

\begin{Remark}
\label{rem:m_and_u_odd} 
We would like to stress that in part (iii) of Theorem
\ref{thm:n_even}, the assumption that at least one of $m$ and $u$ is
even is necessary. This assumption was unfortunately missing in our
earlier preprint \cite{HolmJorgensenDE}. We are grateful to Alex Dugas
for pointing this out to us.

If both $m$ and $u$ are odd then the Calabi-Yau dimension of the
stable category $\stab\, (D_{3m},\frac{s}{3},1)$ cannot be of the
form $u+1$, as would be needed for being a $u$-cluster category.  In
fact, the Calabi-Yau dimension can again be computed using \cite[thm.\
6.1(2)]{Dugas} (\cite[thm.\ 6.1(1)]{Dugas} does not apply since $m$ is
odd). In particular, the Calabi-Yau dimension is of the form $d=2r+1$
and hence an odd number which makes it impossible to be equal to $u+1$
(since $u$ is odd).

This happens despite the fact that, for $m$ and $u$ odd, the stable
module category $\stab\, (D_{3m},\frac{s}{3},1)$ and the $u$-cluster
category of type $D_{3m}$ both have as AR quiver a cylinder of
circumference $s(2m-1)$.  The reason is that under the AR translation
$\tau$, the exceptional vertices form a single orbit in the
$u$-cluster category but two orbits in the stable category.

As an explicit example, consider the case when $m=3$ and $u=3$. 
Then the Calabi-Yau dimension of the stable module 
category $\stab\, (D_{9},\frac{5}{3},1)$ is, according
to \cite[thm.\ 6.1(2)]{Dugas}, of the form
$2r+1$ where $r$ is determined by
$r \equiv - 16^{-1} \,\modulo\, 25 \equiv 14 \,\modulo\, 25$
and $0\le r < 25$. Thus, $\stab\, (D_{9},\frac{5}{3},1)$ 
has Calabi-Yau dimension 29, which is far from the 
Calabi-Yau dimension 4 of the $3$-cluster category of type 
$D_9$. 
\end{Remark}

\begin{Example}

We illustrate our realizability results in type $D_n$ from Theorem
\ref{thm:n_even} by considering the situation for some small values of
$n$. 

Let us first consider type $D_4$. Then parts (ii) and (iii) of Theorem
\ref{thm:n_even} do not apply. From part (i) we get for every
$u\equiv\,3 \,\modulo\, 5$ that the $u$-cluster category of type $D_4$
is triangulated equivalent to $\stab\,(D_4,\frac{3u+1}{5},1)$.

Let us now consider type $D_6$. From part (i) of Theorem
\ref{thm:n_even} we get for every $u\equiv\,7 \,\modulo\, 9$ that
the $u$-cluster category of type $D_6$ is triangulated equivalent
to the category $\stab\,(D_6,\frac{5u+1}{9},1)$.

Moreover, from part (iii) of Theorem \ref{thm:n_even} we also get that
for every $u\equiv\,1 \,\modulo\, 9$ and every $u\equiv\,4 \,\modulo\,
9$ that the $u$-cluster category of type $D_6$ is triangulated
equivalent to $\stab\,(D_{6},\frac{5u+1}{9},1)$.  Hence, for all
$u\equiv\,1 \,\modulo\, 3$, we get the $u$-cluster category of type
$D_6$ as stable module category of a selfinjective algebra.

We remark that the smallest case $u=1$ states that the $1$-cluster
category of type $D_6$ is triangulated equivalent to the stable module
category of the preprojective algebra of type $A_4$. In fact, the
algebra $(D_{6},\frac{2}{3},1)$ is just this preprojective algebra.
This can be considered as the cluster category version of the
statement that the preprojective algebra of type $A_4$ is of cluster
type $D_6$ \cite[sec. 19.2]{GLS_semcan1}.  For more details on the
close connection between preprojective algebras and cluster theory we
refer to \cite{GLS_rigid}.
\end{Example}

\section{Dynkin type $E$}
\label{sec:typeE}

This section proves Theorem E from the introduction.

Asashiba's paper \cite{Asashiba1} gives that if the tree class of the
stable AR quiver is Dynkin type $E$, then there are four families of
representatives of self-injective algebras denoted
\begin{itemize}

  \item  $(E_6,s,1)$,

\medskip

  \item  $(E_6,s,2)$,

\medskip

  \item  $(E_7,s,1)$,

\medskip

  \item  $(E_8,s,1)$,

\end{itemize}
all with $s\ge 1$.  Recall that in type $E$, nonstandard algebras do
not occur.  It follows from \cite[cor.\ 1.7]{BS} that the stable AR
quivers of these algebras are cylinders with the following
circumferences.
\begin{itemize}

  \item  For $(E_6,s,1)$ and $(E_6,s,2)$ the circumference is $11s$.

\medskip

  \item  For $(E_7,s,1)$ the circumference is $17s$.

\medskip

  \item  For $(E_8,s,1)$ the circumference is $29s$.

\end{itemize}

By Subsection \ref{subsec:clusterE}, the AR quiver of the $u$-cluster
category of type $E_6$ is a cylinder or a M\"{o}bius band of
circumference $6u+1$ (this number is independent of $u$ being even or
odd).  So in order for the stable categories $\stab\,(E_6,s,1)$ or
$\stab\,(E_6,s,2)$ to be $u$-cluster categories we need $6u+1 = 11s$.
In particular, this implies
\[
  u\equiv\,-6^{-1}\,\equiv\,-2 \,\modulo\, 11.
\]
Likewise, the AR quiver of the $u$-cluster category of type $E_7$ is a
cylinder of circumference $9u+1$.  So in order for the stable
ca\-te\-go\-ry $\stab\,(E_7,s,1)$ to be a $u$-cluster category we need
$9u+1 = 17s$.  In particular, this implies
\[
  u\equiv\,-9^{-1}\,\equiv\,-2 \,\modulo\, 17.
\]
Finally, the AR quiver of the $u$-cluster category of type $E_8$ is a
cylinder of circumference $15u+1$.  So in order for the stable module
category $\stab\,(E_8,s,1)$ to be $u$-cluster category we need $15u+1
= 29s$.  In particular, this implies
\[
  u\equiv\,-15^{-1}\,\equiv\,-2 \,\modulo\, 29.
\]

Indeed, these conditions turn out also to be sufficient.  The main
result of this section is the following which restates Theorem E from
the introduction.

\begin{Theorem}
\label{thm:E6}
Let $u \geq 1$ be an integer.
\begin{enumerate}

\item If $u\equiv\, -2 \,\modulo\, 11$ then the $u$-cluster category of
  Dynkin type $E_6$ is equivalent as a triangulated category to the
  stable module category $\stab\, (E_6,\frac{6u+1}{11},1)$ if $u$ is
  even, and to the stable module category $\stab\,
  (E_6,\frac{6u+1}{11},2)$ if $u$ is odd.

\medskip

\item If $u\equiv\, -2 \,\modulo\, 17$ then the $u$-cluster category of
  Dynkin type $E_7$ is equivalent as a triangulated category to the
  stable module category $\stab\, (E_7,\frac{9u+1}{17},1)$.

\medskip

\item If $u\equiv\, -2 \,\modulo\, 29$ then the $u$-cluster category of
  Dynkin type $E_8$ is equivalent as a triangulated category to the
  stable module category $\stab\, (E_8,\frac{15u+1}{29},1)$.

\end{enumerate}
\end{Theorem}

\begin{proof}
As in types A and D, the proof is divided into three sections
verifying the conditions in Keller and Reiten's Morita theorem
\cite[thm.\ 4.2]{KellerReiten2}.

\medskip
{\em Calabi-Yau dimension. }
We must show that the relevant stable module categories have
Calabi-Yau dimension $u + 1$, and again we do so using the results by 
Dugas from \cite{Dugas}.

First, consider the algebras $(E_6,\frac{6u+1}{11},1)$; in particular
$u$ is assumed to be even.  Then \cite[thm.\ 6.1(2)]{Dugas}
applies. Note that the invariants occurring there for type $E_6$ are
given by: The frequency $f=\frac{6u+1}{11}$, the Coxeter number
$h_{E_6}=12$, and $m_{E_6} = h_{E_6}-1=11$. The Calabi-Yau dimension of
$\stab\, (E_6,\frac{6u+1}{11},1)$ is then of the form $2r+1$ where
$$r\equiv -(h_{E_6})^{-1} \,\modulo\, fm_{E_6} = -12^{-1} \,\modulo\, (6u+1)
$$
and $0\le r < 6u+1$. Since $u$ is even by assumption we can consider the integer 
$r=\frac{u}{2}$; this clearly satisfies $12r = 6u \equiv -1
\,\modulo\, (6u+1)$,
and $0\le r < 6u+1$. Therefore the Calabi-Yau dimension of $\stab\, (E_6,\frac{6u+1}{11},1)$
is $2r+1 = u+1$, as desired. 

Secondly, consider the algebras $(E_6,\frac{6u+1}{11},2)$; in
particular $u$ is assumed to be odd.  Then we can apply \cite[prop.\
7.4(1)]{Dugas}. The Calabi-Yau dimension of $\stab\,
(E_6,\frac{6u+1}{11},2)$ is then equal to $2r$ where $r \equiv 5\cdot
12^{-1} \,\modulo\, (6u+1)$ and $0< r < 6u+1$. Setting
$r=\frac{u+1}{2}$ (recall that $u$ is odd by assumption) we
immediately get that $12 r = 6(u+1) \equiv 5 \,\modulo\, (6u+1)$ and
hence the Calabi-Yau dimension is $2r = u+1$, as desired.

Thirdly, consider the algebras $(E_7,\frac{9u+1}{17},1)$. The
Calabi-Yau dimension can again be determined by \cite[thm.\
6.1]{Dugas}. The relevant invariants for type $E_7$ are given by: The
frequency $f=\frac{9u+1}{17}$, the Coxeter number $h_{E_7}=18$ with
its variant $h_{E_7}^* = h_{E_7}/2 = 9$, and $m_{E_7} = h_{E_7}-1=17$.
Note that for type $E_7$ both parts of \cite[thm.\ 6.1]{Dugas} can
possibly apply; we shall show that in either case we get $u+1$ as
Calabi-Yau dimension of the stable module category.

In \cite[thm.\ 6.1(1)]{Dugas} the Calabi-Yau dimension 
$d$ satisfies
$$d\equiv 1 - (h_{E_7}^*)^{-1} \,\modulo\, fm_{E_7}
= 1 - 9^{-1} \,\modulo\, (9u+1)
$$
as well as $0 < d \le 9u+1$. Clearly, $d=u+1$ has these properties,
and hence the Calabi-Yau dimension is $u+1$, as desired.

In \cite[thm.\ 6.1(2)]{Dugas} the Calabi-Yau dimension $d$ has the
form $d=2r+1$ where $r \equiv - 18^{-1} \,\modulo\, (9u+1)$ and $0\le
r < 9u+1$. Note that when \cite[thm.\ 6.1(2)]{Dugas} applies then
$2\nmid f =\frac{9u+1}{17}$ which implies that $u$ is even. Setting
$r=\frac{u}{2}$ we immediately see that $18r \equiv 9u \equiv -1
\,\modulo\, (9u+1)$, i.e.  the Calabi-Yau dimension of the stable
category in this case is also $2r+1 = u+1$, as desired.

Finally, consider the algebras $(E_8,\frac{15u+1}{29},1)$. The
arguments for determining the Calabi-Yau dimension by \cite[thm.\
6.1]{Dugas} are very similar to the previous case of $E_7$. The
relevant invariants for type $E_8$ are: The frequency
$f=\frac{15u+1}{29}$, the Coxeter number $h_{E_8}=30$ with its variant
$h_{E_8}^* = h_{E_8}/2 = 15$, and $m_{E_8} = h_{E_8}-1=29$.  Again,
both parts of \cite[thm.\ 6.1]{Dugas} can apply.  In \cite[thm.\
6.1(1)]{Dugas} the Calabi-Yau dimension $d$ satisfies $d\equiv 1 -
15^{-1} \,\modulo\, (15u+1)$ and $0 < d \le 15u+1$. Clearly, $d = u+1$
has these properties, and hence the Calabi-Yau dimension is $u+1$, as
desired.

In \cite[thm.\ 6.1(2)]{Dugas} the Calabi-Yau dimension $d$ has the
form $d=2r+1$ where $r \equiv - 30^{-1} \,\modulo\, (15u+1)$ and $0\le
r < 15u+1$. As before, when \cite[thm.\ 6.1(2)]{Dugas} applies then
$2\nmid f =\frac{15u+1}{29}$ which implies that $u$ is even. Setting
$r=\frac{u}{2}$ we get that $30r \equiv 15u \equiv -1 \,\modulo\,
(15u+1)$, i.e.  the Calabi-Yau dimension of the stable category in
this case is also $2r+1 = u+1$, as desired.

\medskip
{\em $u$-cluster tilting object. }
To find a $u$-cluster tilting object $X$, as in type $D$, let $X$ be
the direct sum of the indecomposable projective modules and the
modules $x_1, \ldots, x_{6}, x_{7}, x_{8}$ whose positions in the
stable AR quiver of the selfinjective algebra are given by Figure
\ref{fig:ZEn2}, with the convention that the summands $x_7$ and $x_8$
only occur in types $E_7$ and $E_8$ as relevant.
\begin{figure}
\[
  \xymatrix @+0.5pc @!0 {
& & & & & {} \ar[dr] & & x_6 \ar[dr] & & \bullet \\
& & & & \ddots \ar[dr] & & x_5 \ar[dr] \ar[ur] & & *{\bullet} 
\ar[dr] \ar[ur] & \\
& & & \ddots \ar[dr] & \cdots \ar[r] & x_3 \ar[dr] \ar[ur] \ar[r] & x_4 
  \ar[r] & *{\bullet} \ar[dr] \ar[ur] \ar[r] & *{\bullet} \ar[r] & \ddots \\
& & \ddots \ar[dr] & & x_2 \ar[ur]\ar[dr] & & *{\bullet} 
\ar[dr] \ar[ur] & & \ddots & \\
& \ddots \ar[dr] & & x_1 \ar[dr] \ar[ur] & & \bullet \ar[dr] \ar[ur] & & 
 \ddots & & \\
\ddots \ar[dr] & & x_7 \ar[dr] \ar[ur] & & *{\circ} \ar[dr] \ar[ur] & & 
  \ddots & & & \\
& x_8 \ar[ur] & & *{\circ} \ar[ur] & & {\circ} & & & & \\
            }
\]
\caption{The indecomposable modules $x_i$ in Dynkin type $E$} 
\label{fig:ZEn2}
\end{figure}
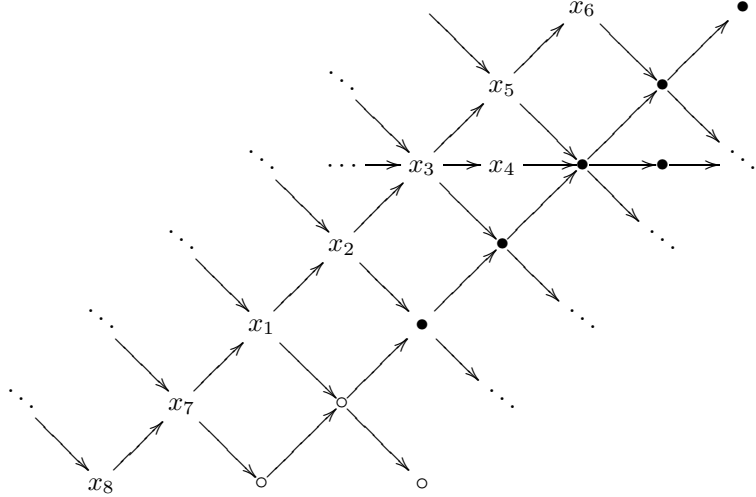

As all algebras in this theorem are standard, their stable module
categories are e\-qui\-va\-lent to the mesh categories of their AR
quivers.
In particular, whether for two objects $v,w$ we have 
$\underline{\Hom}(v,w)\neq 0$ is completely determined by the mesh 
relations. 

For type $E$ and our special choice of object $X$, we get the
following description of the indecomposable objects $t$ such that
$\underline{\Hom}(X,t)\neq 0$ directly from the mesh relations (we
leave the details of the straightforward, though tedious, verification
of these facts to the reader).

In type $E_6$, the objects $t$ are precisely the ones lying
in a trapezium with $X$ as left side and $\tau\Sigma X$ as right 
side; i.e. a trapezium with $X$ as left side, with top side 
containing 4 vertices and bottom side containing 8 vertices
(recall from Section \ref{sec:cluster} that $\Sigma$ is acting
by shifting 6 units to the right and reflecting in the central
line). 

In types $E_7$ and $E_8$ the situation is different.  The
indecomposable objects $t$ such that $\underline{\Hom}(X,t)\neq 0$ are
precisely the ones lying in a parallelogram with $X$ as left side, and
top and bottom sides containing $9$ (for $E_7)$ and $15$ (for $E_8$)
vertices, respectively.

Now we are in a position to show that our chosen object $X$
is indeed a $u$-cluster tilting object. 
We need to describe the objects $t$ with
$\underline{\Hom}(X,\Sigma^it)\neq 0$ for some $i\in\{1,\ldots,u\}$. 
For this purpose, let us consider the regions $H(j)$ of the stable AR quiver
corresponding to indecomposable objects $t$ for which
\begin{equation}
\label{equ:u-cluster}
  \underline{\Hom}(X,\Sigma^{(u+1)-j}t)\neq 0
\end{equation}
where $j$ ranges through $\{ 1, \ldots, u \}$.

For type $E_6$ we suppose that $u\equiv -2 \,\modulo\, 11$.  We have
to distinguish the cases where $u$ is even and odd, respectively.
According to the above description, the regions $H(j)$ look as
follows.

If $u$ is even, then they tile a parallelogram with left side
$\Sigma^{-u}T$ and right side $\tau T$ as in Figure \ref{fig:EnHs}.
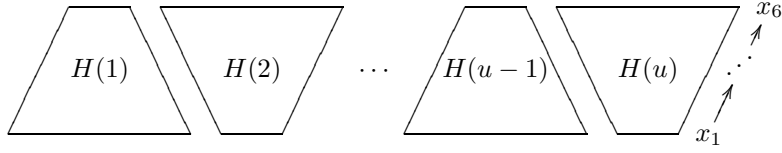
\begin{figure}
\[
  \xymatrix @C-1.05pc @!0 {
  & & *{} \ar@{-}[rr] & & *{} & *{} \ar@{-}[rrrrrr] & & & & & & *{} &&&&*{}\ar@{-}[rr]&&*{}& *{} \ar@{-}[rrrrrr] & & & & & & *{}&x_6\\
  & &&H(1)&&&&&H(2)&&&& \cdots &&&& H(u-1) &&&&& H(u) &&& \adots\ar[ur] \\
  *{} \ar@{-}[uurr] \ar@{-}[rrrrrr] & & & & & & *{} \ar@{-}[uull] & *{} \ar@{-}[uull] \ar@{-}[rr] & & *{} \ar@{-}[uurr] &&&& *{} \ar@{-}[uurr] \ar@{-}[rrrrrr] & & & & & & *{} \ar@{-}[uull] & *{} \ar@{-}[uull] \ar@{-}[rr] & & *{} \ar@{-}[uurr] &x_1\ar[ur]\\
                          }
\]
\caption{The sets $H(i)$ in Dynkin type $E_6$ for $u$ even}
\label{fig:EnHs}
\end{figure}
In particular, the top and bottom sides of this parallelogram contain
$\frac{u}{2}\cdot 12 = 6u$ vertices. But by \cite[cor.\ 1.7]{BS}
(cf. also the remarks at the beginning of this section), the stable
category $\stab\,(E_6,\frac{6u+1}{11},1)$ has precisely $66\cdot
\frac{6u+1}{11} = 6(6u+1)$ indecomposable objects, i.e. the stable AR quiver
(of tree class $E_6$) is identified after $6u+1$ steps.  Hence it
follows that
\begin{center}
$\underline{\Hom}(X,\Sigma^it)=0$ for all $i=1,\ldots,u$ if and only 
if $t \in \add X$.
\end{center}
A very similar argument shows that also
\begin{center}
$\underline{\Hom}(t,\Sigma^iX)=0$ for all $i=1,\ldots,u$ if and only 
if $t \in \add X$.
\end{center}
Thus we have shown that our chosen object $X$ is indeed a $u$-cluster
tilting object in the stable module category $\stab
(E_6,\frac{6u+1}{11},1)$.

If $u$ is odd, then the regions $H(j)$ for $j$ ranging through $\{ 1,
\ldots, u \}$ tile a trapezium with left side $\Sigma^{-u}X$ and right
side $\tau X$ as in Figure \ref{fig:EnHs2}.
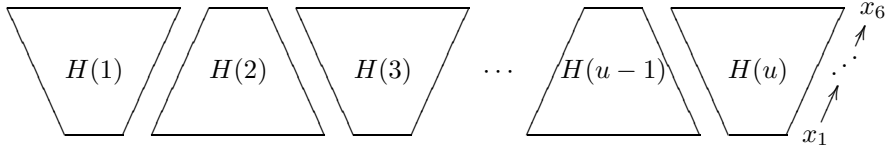
\begin{figure}
\[
  \xymatrix @C-1.1pc @!0 {
*{} \ar@{-}[rrrrrr] & & & & & & *{} & *{} \ar@{-}[rr] & & *{} & *{} \ar@{-}[rrrrrr] & & & & & & *{} &&&&*{}\ar@{-}[rr]&&*{}& *{} \ar@{-}[rrrrrr] & & & & & & *{}&x_6\\
&&&H(1)&&& &&H(2)&&&&&H(3)&&&& \cdots &&&& H(u-1) &&&&& H(u) &&&\adots \ar[ur]&& \\
& & *{} \ar@{-}[rr] \ar@{-}[uull] & & *{} \ar@{-}[uurr] & *{} \ar@{-}[uurr] \ar@{-}[rrrrrr] & & & & & & *{} \ar@{-}[uull] & *{} \ar@{-}[uull] \ar@{-}[rr] & & *{} \ar@{-}[uurr] &&&& *{} \ar@{-}[uurr] \ar@{-}[rrrrrr] & & & & & & *{} \ar@{-}[uull] & *{} \ar@{-}[uull] \ar@{-}[rr] & & *{} \ar@{-}[uurr] & x_1\ar[ur]\\
                         }
\]
\caption{The sets $H(i)$ in Dynkin type $E_6$ for $u$ odd}
\label{fig:EnHs2}
\end{figure}
In particular, the top side of this trapezium contains
$\frac{u-1}{2}\cdot 12 + 8 = 6u + 2$ vertices, and the bottom side
contains $6u - 2$ vertices. In total, this trapezium then contains
$36u$ vertices (e.g.\ note that each of the $u$ smaller trapeziums
with top and bottom sides of length 4 and 8 contains 36 vertices).
But by \cite[cor.\ 1.7]{BS} (cf. also the remarks at the beginning of
this section), the stable category $\stab\,(E_6,\frac{6u+1}{11},2)$
has precisely $6(6u+1)=36u+6$ indecomposable objects. Thus, the above
trapezium fills precisely the region between the parts which become
identified in the stable AR quiver.  Now we can argue as above to
deduce that $X$ is indeed a $u$-cluster tilting object in $\stab
(E_6,\frac{6u+1}{11},2)$.

For types $E_7$ and $E_8$ we suppose that $u\equiv -2 \,\modulo\, 17$
and $u\equiv -2 \,\modulo\, 29$, respectively. Similarly to the above
considerations in type $E_6$ when $u$ is even, the regions $H(j)$ of
indecomposable objects $X$ satisfying equation \eqref{equ:u-cluster}
tile a parallelogram with top and bottom rows containing $9u$ (for
$E_7)$ and $15u$ (for $E_8$) vertices.  A sketch would resemble Figure
\ref{fig:DnHs}.  On the other hand, again by \cite[cor.\ 1.7]{BS}, the
number of indecomposable objects for the stable categories
$\stab(E_7,\frac{9u+1}{17},1)$ and $\stab(E_8,\frac{15u+1}{29},1)$
occurring in the theorem are $119\cdot \frac{9u+1}{17} = 7(9u+1)$ and
$232\cdot \frac{15u+1}{29} = 8(15u+1)$, respectively. Hence, the AR
quivers of these stable categories are identified after $9u+1$ (for
$E_7$) and $15u+1$ units (for $E_8$).  From the sizes of the
parallelograms given above it then follows (just as in the previous
cases) that $X$ is a $u$-cluster tilting object of the relevant stable
categories in types $E_7$ and $E_8$.

The desired fact that the stable endomorphism algebra
$\underline{\End}(X)$ is $kE_n$ is proved verbatim as in type $D$
above.

\medskip
{\em Vanishing of negative self-extensions. }
We must show $\underline{\Hom}(X,\Sigma^{-i} X) = 0$ for $i=1, \ldots,
u-1$.  We have described above the regions in the stable AR quiver where the
modules $t$ are located for which $\underline{\Hom}(X,t)\neq 0$; let
us again denote them by $H(X)$.  These regions $H(X)$ are certain
trapeziums (for $E_6$) or parallelograms (for $E_7$ and $E_8$).

For type $E_6$ we get a situation for which a sketch would resemble
the one above.  The situation for $E_7$ and $E_8$ is completely
analogous, but using parallelograms instead of trapeziums.

As in type $D$, the only way we could fail to get
$\underline{\Hom}(X,\Sigma^{-i} X) = 0$ would be if we took $i$ so
large that $\Sigma^i X$ made it all the way around the stable AR quiver and
reached the forbidden region from the right.

However, this does not happen: In fact, left of $\Sigma^{-(u-1)}X$ we
have the parallelogram (resp. trapezium) $\Sigma^{-u}H(X)$ before
objects get identified in the stable AR quiver. Hence
$\underline{\Hom}(X,\Sigma^{-i} X) = 0$ for $i = 1, \ldots, u - 1$ as
desired. Note that it is crucial that the maximum value for $i$ here
is $u-1$; of course, we have that $\underline{\Hom}(X,\Sigma^{-u} X)
\neq 0$.
\end{proof}

\end{document}